\documentclass[11pt]{amsart}
\usepackage{bbm}
\usepackage{xcolor}
\usepackage[normalem]{ulem}
\usepackage{subcaption}
\usepackage{tikz}
\usetikzlibrary{patterns}

\usepackage[bookmarks,colorlinks]{hyperref}
\usepackage{amssymb}
\theoremstyle{plain}
\newtheorem{thm}{Theorem}

\newtheorem{lem}[thm]{Lemma}

\theoremstyle{definition}

\theoremstyle{remark}
\newtheorem*{rem*}{Remark}
\newtheorem{rem}{Remark}

\newcommand{\Rd}{{\mathbb{R}^d}}

\newcommand{\N}{{\mathbb{N}}}

\newcommand{\R}{\mathbb{R}}

\newcommand{\EEE}{\mathcal{E}}
\newcommand{\AAA}{\mathcal{A}}
\newcommand{\al}{\alpha}
\def \tp{\tilde{p}}
\newcommand{\indyk}{\mathbf{1}}
\date{\today}
\author[K.~Bogdan]{Krzysztof Bogdan}
\address{Faculty of Pure and Applied Mathematics, Wroc\l{}aw University of Science and Technology, Wyb. Wyspia\'nskiego 27, 50-370 Wroc\l{}aw, Poland.}
\email{krzysztof.bogdan@pwr.edu.pl}
\author[T.~Jakubowski]{Tomasz Jakubowski}
\address{Faculty of Pure and Applied Mathematics, Wroc\l{}aw University of Science and Technology, Wyb. Wyspia\'nskiego 27, 50-370 Wroc\l{}aw, Poland.}
\email{tomasz.jakubowski@pwr.edu.pl}
\author[J.~Lenczewska]{Julia Lenczewska}
\address{Faculty of Pure and Applied Mathematics, Wroc\l{}aw University of Science and Technology, Wyb. Wyspia\'nskiego 27, 50-370 Wroc\l{}aw, Poland.}
\email{julia.lenczewska@pwr.edu.pl}
\author[K.~Pietruska-Pa\l uba]{Katarzyna Pietruska-Pa\l uba}
\address{Institute of Mathematics, University of Warsaw, ul. Banacha 2, 02-097 Warsaw, Poland.}
\email{kpp@mimuw.edu.pl}
\thanks{The first named author was partially supported by the NCN grant 2017/27/B/ST1/01339.  The second and third named authors were partially supported by the NCN grant
2015/18/E/ST1/00239. The fourth named author was partially supported by the NCN grant  2018/31/B/ST1/03818.}



\begin{document}
\title[Hardy inequalities]{Optimal Hardy inequality for the fractional Laplacian on $L^p$}

\begin{abstract}
For the fractional Laplacian we give Hardy inequality which is optimal in $L^p$ for $1<p<\infty$. As an application, we explicitly characterize the contractivity of the corresponding Feynman-Kac semigroups on $L^p$.
\end{abstract}
\maketitle

\section{Introduction}
\label{sec:int}
Hardy inequalities
 are of paramount importance in harmonic analysis, functional analysis, partial differential equations, potential theory and probability.
They are applied to
embedding theorems, Gagliardo--Nirenberg interpolation inequalities and in real interpolation theory, see Chua \cite{MR2138501}, 
 Ka{\l}amajska and Pietruska-Pa{\l}uba \cite{MR2852869}.
They 
yield contractivity of operator semigroups,
a priori estimates, existence and regularity results  for solutions of PDEs, 
plus their asymptotics 
and qualitative properties, see, e.g.,  Maz'ya \cite{MR2777530}, Arendt, Goldstein and Goldstein \cite{MR2259099},  Barras and Goldstein \cite{MR742415}, and Vazquez and Zuazua \cite{MR1760280}.
The important connection betweeen Hardy-type inequalities and superharmonic functions was exploited, e.g., by Ancona \cite{MR856511}, Fitzsimmons \cite{MR1786080}, Bogdan and Dyda \cite{MR2663757}, Dyda \cite{Dyda12}, 
Devyver, Fraas and Pinchover \cite{MR3170212},
Bogdan, Dyda and Kim \cite{MR3460023}, Bogdan, Jakubowski, Grzywny and Pilarczyk \cite{MR3933622}.
In particular, the
inequalities are connected to sharp estimates of the heat kernel of $\Delta^{\alpha/2} + \kappa|x|^{-\alpha}$  \cite{MR3933622}, see also Calvaruso, Metafune, Negro and Spina \cite{MR4043021}.

For an account of
the history of Hardy-type inequalities we refer to Opic and Kufner \cite{MR1069756}.
The subject was initiated in 1920, when Hardy \cite{MR1544414} discovered that
\begin{equation}\label{e:hardycls}
\int_{0}^{\infty} \left[u'(x)\right]^2 dx \geq \frac{1}{4} \int_0^{\infty} \frac{u(x)^2}{x^2} dx,\end{equation}
for absolutely continuous functions $u$ such that $u(0)=0$ and $u' \in L^2(0,\infty)$.
The classical Hardy inequality in $\Rd$ for $d \geq 2$
is
\begin{equation}\label{e:hardycl}
\int_{\Rd} |\nabla u(x)|^2 \ge \frac{(d-2)^2}{4}\int_{\Rd} \frac{u(x)^2}{|x|^2} dx, \quad u\in L^2(\Rd).
\end{equation}
Here the left-hand side of \eqref{e:hardycl} is considered infinite if the distributional gradient of $u$ is not a square-integrable function,  see, e.g., \cite[(30) and (32)]{MR3460023} for this formulation.

In 2000 Fitzsimmons \cite{MR1786080} 
proved an abstract analogue of \eqref{e:hardycl}, in which the Dirichlet integral appearing on the left-hand side of \eqref{e:hardycl} is replaced
by a general symmetric Dirichlet form $\EEE$ in the sense of  Fukushima, Oshima, Takeda \cite{MR2778606}.
The rule stemming from \cite{MR1786080} is the following: If $\mathcal{L}$ is the generator of the
form
and function $h$ is superharmonic, i.e., $h\ge 0$ and $\mathcal{L}h\le 0$, then
$\EEE(u,u)\ge -\int u^2 \mathcal{L}h/h$.

The paper \cite{MR3460023} gives similar results in the setting of symmetric transition densities, with
explicit construction of the
function $h$, Riesz' charge
$-\mathcal{L}h$,
and  the Hardy weight, or Fitzsimmons' ratio,
$-\mathcal{L}h/h$.
The resulting
Hardy inequalities, in fact, Hardy identities, in some cases are optimal in the sense of large weight and large functional domain, desirably the whole of $L^2$.

The present work extends part of the results of \cite{MR3460023} to the setting of $L^p$ spaces with arbitrary $p\in (1,\infty)$.
We focus on integral forms
related to the semigroup of fractional Laplacian and give optimal inequalities in this case. We also show that the inequalities lead to optimal contractivity results for related operator semigroups on $L^p$.
From our presentation it should also be evident that the approach applies to more general sub-Markovian semigroups.

\subsection{Sobolev-Bregman forms}
Let  $d \in \N$ and $0<\al<2$.
We consider the fractional Laplacian,
$$
\Delta^{\alpha/2}u(x) := -(-\Delta)^{\alpha/2}u(x):=\underset{\epsilon \to 0+}{\lim} \int_{|y-x|>\epsilon} \left(u(y)-u(x)\right) \nu(x,y) \,dy,\quad x\in \mathbb R^d.
$$
Here, say, $u\in C_c^{2}(\mathbb R^d)$,
\[
\nu(x,y)=\AAA_{d,-\al}|y-x|^{-d-\al}, \quad \quad x, y \in \Rd,
\]
and $\AAA_{d,-\al}=
2^{\al}\Gamma\big((d+\al)/2\big)\pi^{-d/2}/|\Gamma(-\al/2)|$.
We further consider
\begin{equation}\label{e:def-quad-form}
\EEE
{[u]}:= \EEE(u,u):= \, \frac{1}{2} \int_{\Rd} \int_{\Rd} (u(x)-u(y))^2 \nu(x,y) \,dy \,dx,
\end{equation}
 defined for every (Borel measurable) $u: \Rd \to \R$.
The natural domain $\mathcal D(\mathcal E)$ of the form
consists of those functions $u \in L^2(\Rd)$
 for which $\EEE[u]<\infty$.
By \cite[Proposition 5]{MR3460023},
for all $0<\alpha<d \wedge 2$, $0 \leq \beta \leq d-\alpha$,
and
$u \in L^2(\Rd)$, we have the Hardy-type identity
\begin{align}\label{e:p2}
&\mathcal{E}{[u]}=
\kappa_{\beta}
\int_{\R^d} \frac{u(x)^2}{|x|^\alpha}\,dx
 + \frac{1}{2} \int_{\R^d}\!\int_{\R^d}
\left[\frac{u(x)}{h(x)}-\frac{u(y)}{h(y)}\right]^2
h(x)h(y) \nu(x,y)
\,dy\,dx,
\end{align}
where $h(x):=|x|^{-\beta}$, and
\begin{equation}\label{e.dkb}
\kappa_{\beta}=\frac{2^\alpha\Gamma\big(\frac{\beta+\alpha}{2}\big) \Gamma\big(\frac{d-\beta}{2}\big)}{\Gamma\big(\frac{\beta}{2}\big)\Gamma\big(\frac{d-\beta-\alpha}{2}\big)}.
\end{equation}
Before \cite{MR3460023}, the identity \eqref{e:p2} was given by Frank, Lieb and Seiringer in \cite[(4.3)]{MR2425175} for functions $u\in C^\infty_0(\Rd\setminus\{0\})$ and $\beta \in [0, (d-\alpha)/2]$. We note that the
 coefficient $\kappa_\beta$
is increasing on $[0,(d-\alpha)/2]$, symmetric with respect to $(d-\alpha)/2$ and takes on the maximal value at
$\beta=(d-\alpha)/2$, which is
$$
\kappa_{(d-\alpha)/2}=2^\alpha\Gamma\left(\frac{d+\alpha}{4}\right)^2
\Gamma\left(\frac{d-\alpha}{4}\right)^{-2},
$$
see \cite[Lemma 3.2]{MR2425175} or \cite[p.~237]{MR3460023}.
Correspondingly, the following optimal fractional Hardy inequality holds
\begin{equation}\label{e:hardy-quad}
 \mathcal{E}{[u]}\geq
\kappa_{(d-\alpha)/2}
\int_{\R^d} \frac{u(x)^2}{|x|^\alpha}\,dx\quad \text{ for all } \quad  u \in L^2(\Rd).
 \end{equation}
The inequality is also known as Hardy-Rellich inequality and  was proved by Herbst \cite[(2.6)]{MR0436854}, Beckner \cite[Theorem 2]{MR1254832} and Yafaev \cite{MR1717839}.
We will propose an optimal analogue of the inequality appropriate for $L^p(\Rd)$. To this end
for  $p \in (1, \infty)$ and (Borel measurable) $u: \Rd \to \R$ we define the Sobolev-Bregman form, or the $p$-form,
\begin{equation}\label{e:dEp}
\EEE_p[u] :=
\frac{1}{2}
\int_{\Rd} \int_{\Rd} (u(x)-u(y)) (u(x)^{\langle p - 1 \rangle}  -u(y)^{\langle p - 1 \rangle} ) \nu(x,y) \,dy \,dx.
\end{equation}
Here and below we use the notation $$a^{\langle k \rangle} := \left|a \right|^k \operatorname{sgn} a,\quad a, k \in \R,$$
where $0^k=0$.
The $p$-form
is well defined since  the integrand in \eqref{e:dEp} is nonnegative for every $u$.
In fact, to compare $\EEE_p$ with $\EEE=\EEE_2$, we recall that
\begin{equation}\label{e.cHk}
4(p-1)p^{-2}(b^{\langle p/2\rangle}-a^{\langle p/2\rangle})^2\leq (b-a)(b^{\langle p-1\rangle}-a^{\langle p-1\rangle})\leq 2(b^{\langle p/2\rangle}-a^{\langle p/2\rangle})^2.
\end{equation}
The inequality
holds true for all $p\in (1,\infty)$ and $a,b\in \R$, see, e.g., Liskevich, Perelmuter and Semenov \cite[Lemma 2.1]{MR1407327}.
By \eqref{e:hardy-quad}
and \eqref{e.cHk},
\begin{equation}\label{e:subgoal}
\EEE_p[u] \geq 4(p-1)p^{-2} \kappa_{(d-\alpha)/2}
\int_{\Rd} \frac{\left| u(x) \right|^p} {|x|^{\alpha}} \,dx, \quad u \in L^p(\Rd).
\end{equation}
We remark that the inequality \eqref{e:subgoal} is given in Cialdea and Maz'ya \cite[p. 231]{MR3235527}.
Our goal is to improve the constant.
To this end, inspired by Bogdan, Dyda and Luks \cite[(9)]{MR3251822},
we consider the \textit{Bregman divergence}:
\begin{equation}
F_p(a,b) := |b|^p - |a|^p - pa^{\langle p - 1 \rangle} (b-a), \quad a, b \in \R.
\end{equation}
It is the second-order Taylor remainder of the convex function $\R\ni x\mapsto |x|^p$, so $F_p(a,b) \geq 0$. For instance, $F_2(a,b)= (b-a)^2$.
The function $F_p$  may be used to
quantify the regularity of functions in integral forms generalizing
\eqref{e:def-quad-form}. Indeed,
the \textit{symmetrization} of $F_p$ is
\begin{equation}\label{e:fp}
\frac12 (F_p(a,b) + F_p(b,a)) =\frac{p}{2} (b-a)(b^{\langle p - 1 \rangle} -a^{\langle p - 1 \rangle} ),
\end{equation}
which is the expression in the definition  of $\mathcal E_p$, up to the factor of $p$.
In passing we refer the reader to Bogdan, Grzywny, Pietruska-Pa{\l}uba and Rutkowski \cite{2020arXiv200601932B} for references to applications of Bregman divergence in analysis, statistical learning and optimization, and to Bogdan and Wi\c{e}cek \cite{bogdan2021Burkholder} for a martingale connection.

\subsection{Main results}
For $\beta\in \R$ we denote $h_\beta(x) := \left|x \right|^{-\beta}$, $x\in \Rd$.
Of course, for $a\in \R$,
\begin{equation}\label{e:hbeta}
h_{\beta}(x)^a = h_{a\beta}(x), \quad x\in \Rd.
\end{equation}
We propose the following Hardy-type identity for the Sobolev-Bregman forms.
\begin{thm}\label{thm:Theorem1}
If $0<\alpha<d \wedge 2$, $0 \leq \beta \leq (d-\alpha) \wedge (d-\alpha)/(p-1)$, $h=h_\beta$ and   $u\in L^p(\mathbb R^d)$, then
\begin{align}\label{e:main}
\EEE_p{[u]} &= \frac{\kappa_{(p-1)\beta}+(p-1)\kappa_{\beta}}{p} \int_{\mathbb{R}^d} \frac{\left|u(x)\right|^p}{\left|x\right|^{\alpha}} dx \\[2mm]
&\quad  +\frac{1}{p}\int_{\Rd} \int_{\Rd}  F_p\left(\frac{u(x)}{h(x)}, \frac{u(y)}{h(y)}\right) h(x)^{p-1}h(y)  \nu(x,y) \,dy \,dx. \nonumber
\end{align}
\end{thm}
\noindent
In particular, for $\beta=(d-\alpha)/p$ we obtain
\begin{align}\label{e:ip}
\EEE_p{[u]} &= \kappa_{(d-\alpha)/p} \int_{\Rd} \frac{|u(x)|^p}{|x|^{\alpha}} dx \\
&\quad+ \frac{1}{p}\int_{\Rd} \int_{\Rd} F_p\left(\frac{u(x)}{h(x)}, \frac{u(y)}{h(y)}\right) h(x)^{p-1}h(y) \nu(x,y) \,dy \,dx,\nonumber
\end{align}
and, of course,
\begin{align}
\EEE_p{[u]}&\geq \kappa_{(d-\alpha)/p} \int_{\Rd} \frac{|u(x)|^p}{|x|^{\alpha}} \,dx
\quad \text{ for all } \quad u \in L^p(\Rd).
\label{e.Hp}
\end{align}
The results extend the Hardy identities from \cite{MR3460023} and the ground-state representations of
Frank, Lieb and Seiringer in \cite[Proposition 4.1]{MR2425175}, see also  Frank and Seiringer \cite{MR2469027} and Beckner \cite{MR2984215}.
The proofs of \eqref{e:main} and \eqref{e:ip} are given in Section~\ref{sec:Hi}.
In Lemma~\ref{thm:lem7} of Section~\ref{sec:Hi}
we also show that \eqref{e.Hp} improves \eqref{e:subgoal}, namely for $p\neq 2$ we have
\begin{equation}\label{e.i}
\kappa_{(d-\alpha)/p}>\frac{4(p-1)}{p^2} \,
\frac{2^\alpha\Gamma\left(\frac{d+\alpha}{4}\right)^2}
{\Gamma\left(\frac{d-\alpha}{4}\right)^{2}}.
\end{equation}
\begin{figure}
\includegraphics[width=12cm]{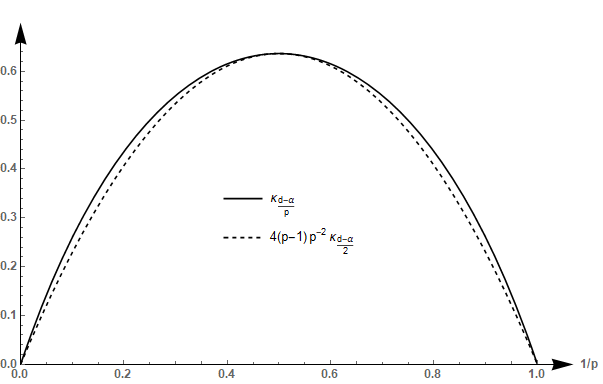}
\caption{Illustration of \eqref{e.i} for $d=3, \alpha=1$}\label{fig:comp}
\end{figure}
Figure \ref{fig:comp} compares both sides of \eqref{e.i}, by showing $\kappa_{(d-\alpha)/p}$ of \eqref{e.Hp} (solid line) and $4(p-1)p^{-2}\kappa_{(d-\alpha)/2}$ of \eqref{e:subgoal} (dashed line) as functions of $1/p\in (0,1]$ for $d=3$ and $\alpha=1$.
Here is a statement deeper than \eqref{e.i}, which we prove in Section~\ref{s:opt}.
\begin{thm}\label{thm:Theorem2}
The constant in \eqref{e.Hp} is sharp.
\end{thm}
The result was previously known only for $p=2$; we again refer to \cite[(2.6)]{MR0436854}, \cite[proof of Theorem 2]{MR1254832} or \cite{MR1717839}.

For the sake of comparison we recall
the fractional Hardy-Sobolev inequality from
\cite[Theorem 1.1]{MR2469027}:
\begin{equation}\label{e:frank}
 \int_{\Rd}\int_{\Rd} \frac{|u(x)-u(y)|^p}{|x-y|^{d+ps} } \,dx\,dy
\geq \mathcal C_{d,s,p} \int_{\R^d} \frac{|u(x)|^p}{|x|^{ps}}\,dx.
\end{equation}
The constant
$C_{d,s,p}$ is known and sharp, too.
Of course,
the forms in \eqref{e:dEp} and  \eqref{e:frank}
are different generalizations of the quadratic form in \eqref{e:def-quad-form}. Correspondingly, the inequalities \eqref{e.Hp} and \eqref{e:frank} are different optimal generalizations of the fractional Hardy inequality \eqref{e:hardy-quad}.
In passing we refer the reader to \cite[Section 6]{2020arXiv200601932B} for an additional insight into the comparison of \eqref{e.Hp} and \eqref{e:frank} and to Frank and Seiringer \cite{MR2723817} for a version of \eqref{e:frank} for half-spaces.

Not unexpectedly, our main application of
Theorem~\ref{thm:Theorem1} and \ref{thm:Theorem2} is to the $L^p$ contractivity of the Feynman-Kac semigroup $\tilde P_t$ generated by $\Delta^{\alpha/2}+\kappa_\delta |x|^{-\alpha}$ with $\kappa_\delta$ given by \eqref{e.dkb}.
\begin{thm}\label{t.znHp}
Let $0<\alpha<2\wedge d$, 
$\delta\in[0,(d-\alpha)/2]$,
$1<p<\infty$ and $0<t<\infty$. The operator $\tilde P_t$ is a contraction on $L^p(\Rd)$ if and only if $\kappa_\delta \le \kappa_{(d-\alpha)/p}$.
\end{thm}
The proof of this characterization is given in Section~\ref{s.ap}, see also Remark~\ref{r.n} there. 
The result is analogous to the classical case ($\alpha=2$), where the operator $\Delta + \kappa|x|^{-2}$ generates a contraction semigroup on $L^p(\R^d)$ if and only if  $\kappa\le \kappa_{(d-2)/p}= (d-2)^2 (p-1)p^{-2}$,
see Kovalenko, Perelmuter and Semenov \cite{MR622855}, see also Liskevich and Semenov \cite[Theorems 1 and 2]{MR1409835} and \cite[Corollary 1.2]{MR2259099}. 
Theorem~\ref{t.znHp} 
is illustrated by Figure~\ref{fig:thm3}.
Notably,   if $\alpha=2$  then $\kappa_\beta=\beta(d-2-\beta)$ and \eqref{e.i} becomes equality, so we have improved Hardy inequality 
for the \textit{nonlocal} operator $\Delta^{\alpha/2}$ but not for the \textit{local} operator $\Delta$.
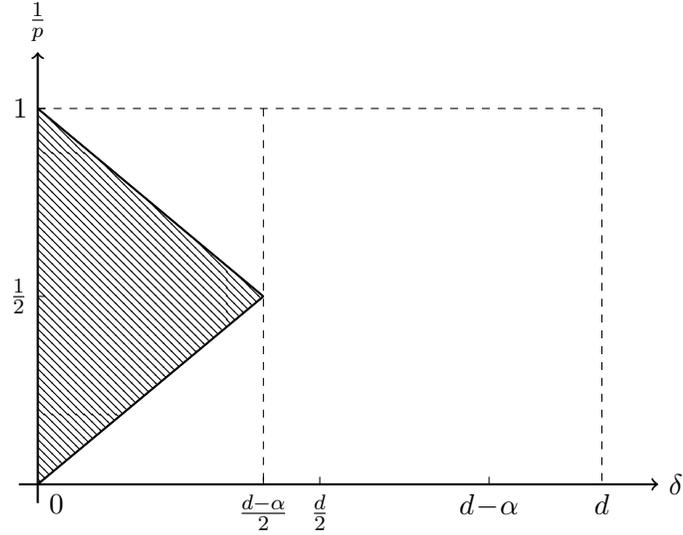
\begin{figure}
\begin{tikzpicture}[
    scale=5,
    axis/.style={thick, ->},
    important line/.style={thick},
    dashed line/.style={dashed, thin},
    ]
    \draw[axis] (-0.05,0)  -- (1.65,0) node(xline)[right]{$\delta$};
    \draw[axis] (0,-0.05) -- (0,1.15) node(yline)[above] {$\frac{1}{p}$};
		\draw[dashed line] (0,1) -- (1.5,1);
		\draw[dashed line] (1.5,1) -- (1.5,0);
		\draw[important line] (0,1) -- (0.6,0.5);
		\draw[important line] (0,0) -- (0.6,0.5);
		\draw[dashed line] (0.6,1) -- (0.6,0);
		\draw [pattern=north west lines] (0,0) -- (0.6,0.5) -- (0,1);
		\draw[-] (0.75,0.0) -- (0.75,0.02);
		\draw[-] (0.6,0.0) -- (0.6,0.02);
		\draw[-] (1.2,0.0) -- (1.2,0.02);
		\draw[-] (0.0,0.5) -- (0.02,0.5);
		\draw[-] (0.0,1) -- (0.02,1);
		\draw[-] (0,0) -- (0.6,0) coordinate[label = {below:$\frac{d-\alpha}{2}$}];
		\draw[-] (0,0) -- (0.75,0) coordinate[label = {below:$\frac{d}{2}$}];
		\draw[-] (0,0) -- (1.5,0) coordinate[label = {below:$d$}];
    \draw[-] (0,0) -- (1.2,0) coordinate[label = {below:$d\!-\!\alpha$}];
		\draw[-] (0,0) -- (0,0.5) coordinate[label = {left:$\frac{1}{2}$}];
		\draw[-] (0,0) -- (0,1) coordinate[label = {left:$1$}];
    \draw[-] (0,0) -- (0.05,0) coordinate[label = {below:$0$}];
    \end{tikzpicture}
		\caption{The shaded area shows the values of $1/p$, for which, given $\delta\in [0,(d-\alpha)/2]$ and $\Delta^{\alpha/2}+\kappa_{\delta}|x|^{-\alpha}$,
$\tilde P_t$ are contractive on $L^p(\Rd)$.}\label{fig:thm3}
		\end{figure}		

Here is the last theorem of the paper.
\begin{thm}\label{t.bd}
Let $0<\alpha<2\wedge d$, $\delta\in[0,(d-\alpha)/2]$, $1<p<\infty$ and $0<t<\infty$. The operator $\tilde P_t$ is bounded on $L^p(\Rd)$ if and only if $\delta < {d/p^*}$, where $p^*=\max\{p,p/(p-1)\}$.
\end{thm}
The result is proved in Section~\ref{s.ap}. 
As we shall see, the result is a consequence of the estimates of $\tilde P_t$ given in \cite{MR3933622}.
An illustration of Theorem~\ref{t.bd} is given in Figure~\ref{fig:thm4}.
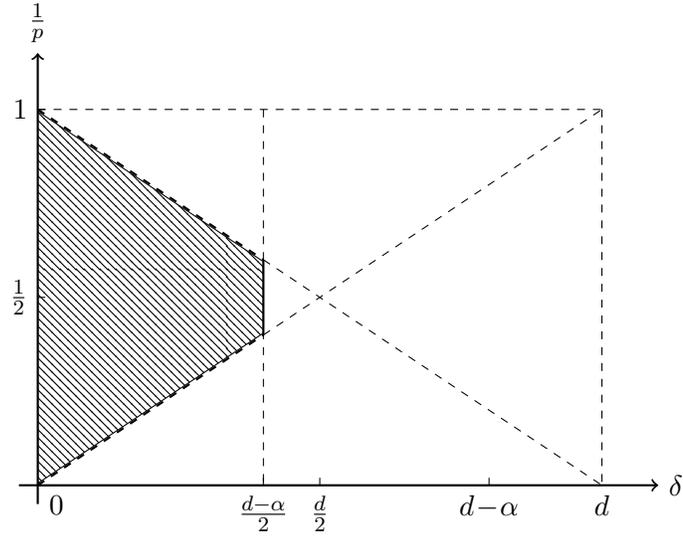
\begin{figure}
			\centering
				\begin{tikzpicture}[
    scale=5,
    axis/.style={thick, ->},
    important line/.style={thick},
    dashed line/.style={dashed, thin},
		dashed important line/.style={dashed, thick},
    ]
    \draw[axis] (-0.05,0)  -- (1.65,0) node(xline)[right]{$\delta$};
    \draw[axis] (0,-0.05) -- (0,1.15) node(yline)[above] {$\frac{1}{p}$};
    \draw[dashed line] (0.6,0.4) -- (1.5,1);
    \draw[dashed line] (0.6,0.6) -- (1.5,0);
		\draw[dashed line] (0,1) -- (1.5,1);
		\draw[dashed line] (1.5,1) -- (1.5,0);
		\draw[important line] (0.6,0.4) -- (0.6,0.6);
		\draw[dashed line] (0.6,1) -- (0.6,0);
		\draw[-] (0.75,0.0) -- (0.75,0.02);
		\draw[-] (0.6,0.0) -- (0.6,0.02);
		\draw[-] (1.2,0.0) -- (1.2,0.02);
		\draw[-] (0.0,0.5) -- (0.02,0.5);
		\draw[-] (0.0,1) -- (0.02,1);
		\draw[-] (0,0) -- (0.6,0) coordinate[label = {below:$\frac{d-\alpha}{2}$}];
		\draw[-] (0,0) -- (0.75,0) coordinate[label = {below:$\frac{d}{2}$}];
		\draw[-] (0,0) -- (1.5,0) coordinate[label = {below:$d$}];
    \draw[-] (0,0) -- (1.2,0) coordinate[label = {below:$d\!-\!\alpha$}];
		\draw[-] (0,0) -- (0,0.5) coordinate[label = {left:$\frac{1}{2}$}];
		\draw[-] (0,0) -- (0,1) coordinate[label = {left:$1$}];
    \draw[-] (0,0) -- (0.05,0) coordinate[label = {below:$0$}];	
		\draw[dashed important line] (0,1) -- (0.6,0.6);
		\draw[dashed important line] (0,0) -- (0.6,0.4);
		\draw [pattern=north west lines] (0.00,0.005) -- (0.6,0.405) -- (0.6,0.595) -- (0,0.995);
    \end{tikzpicture}
		\caption{The shaded area shows the values of $1/p$, for which, given  $\delta\in [0,(d-\alpha)/2]$ and $\Delta^{\alpha/2}+\kappa_{\delta}|x|^{-\alpha}$,
$\tilde P_t$ are bounded on $L^p(\Rd)$.}\label{fig:thm4}
	\end{figure}

The structure of the paper is as follows. In Section~\ref{sec:Pre} we give preliminaries and ponder  definitions. The proof of Theorem~\ref{thm:Theorem1} is given in Section~\ref{sec:Hi}. The  proof of Theorem~\ref{thm:Theorem2} is in Section~\ref{s:opt}.
In Section~\ref{s.ap} we prove Theorem~\ref{t.znHp} and Theorem~\ref{t.bd}.
In Section~\ref{s.d} 
we point out to related results in the literature and broader perspectives.

\textbf{Acknowledgements.} We thank Rupert Frank, Damir Kinzebulatov, Bart\l{}omiej Dyda, Jerome Goldstein, Kamil Kaleta, Mateusz Kwa\'{s}nicki, Adam Nowak, Sergey Pirogov, R\'ene Schilling and Karol Szczypkowski for helpful comments, discussions and references. We thank Yuli Semenov for an inspiring lecture in September 2018 in Wroc\l{}aw.

\subsection{Further discussion}\label{s.d}

The main results of the paper are the optimal Hardy identity \eqref{e:ip} and inequality \eqref{e.Hp}. In this section, however, we try to give a broader perspective for the considered Sobolev-Bregman forms and point out to connections and inspirations. It is well known that
the ramification of the potential theory of \textit{symmetric} Markovian processes and semigroups by means of the theory of Dirichlet forms on $L^2$ spaces turned out extremely succesful \cite{MR2778606}, Ma and R\"{o}ckner \cite{MR1214375}. Therefore many authors worked to extend it to the case of the $L^p$ spaces. The development is related to the fact that
$\langle -A f, f^{p-1}\rangle\ge 0$ for $f$ in the domain of the generator,
see \eqref{e.pfag}
below. We recall that for $p\in (1,\infty)$ the dual space of $L^p$ is, of course,  $L^{p/(p-1)}$ and for $u\in L^p$ we have $u^{\langle p-1\rangle}\in L^{p/(p-1)}$, and $\|u\|_p^{p}=\int |u|^p=\int |u^{\langle p-1\rangle}|^{p/(p-1)}=\int u^{\langle p-1\rangle}u$. Therefore $\|u\|_p^{2-p} u^{\langle p-1\rangle}$ yields a linear functional on $L^p$ appropriate for testing the dissipativity of generators in the Lumer--Phillips theorem,
see, e.g., Pazy \cite[Section 1.4]{MR710486} or \cite[p. 3]{MR3235527}. For the semigroups generated by local operators we refer to Langer and Maz'ya \cite{MR1694522}, Sobol and Vogt \cite[Theorem 1.1]{MR1923627} and the monograph 
\cite{MR3235527},
see also
the historical comments in Section~4.7 therein.

As we shall experience first-hand in Section~\ref{s.ap}, the forms $\mathcal E_p$ capture the evolution of $L^p$ norms of functions upon the action of operator semigroups. The connection is well known, since at least 
the paper of
Varopoulos \cite[(1,1)]{MR803094}.
It was then used, e.g., to study 
perturbations of semigroups on $L^p$ 
in \cite[Theorem 3.2]{MR1407327}
and \cite{MR1409835}.
For nonlocal operators we refer the reader to Farkas, Jacob and Schilling \cite[(2.4)]{MR1808433} and to the monograph of Jacob \cite[(4.294)]{MR1873235}.
As a rule the authors use the intermediary $L^2$ setting, relying on inequalities similar to \eqref{e.cHk} and to
\eqref{e.cHkf} below. Therefore the resulting inequalities
do not allow for optimal constants, certainly in the case of nonlocal generators, cf. \eqref{e:subgoal} and \eqref{e.Hp}; see also \cite[p. 231]{MR3235527}.
Suboptimal constants may lead to suboptimal qualitative results,
as shall be evident from the proof of Theorem~\ref{t.znHp}.
We avoid the inaccuracy of \eqref{e.cHk} by systematically using the Bregman divergence -- we work \emph{intrinsically} in $L^p$ and therefore obtain the optimal constants and the optimal range of exponents in Theorems \ref{thm:Theorem2} and \ref{t.znHp}.

We remark that \eqref{e.cHkf} and \eqref{e.pfag} can be traced back to the papers \cite{MR803094} and Liskevich, Semenov \cite{MR1160303}. In our notation it reads there as $\langle -A f, f^{\langle p-1\rangle}\rangle\approx \langle -A^{1/2}f^{\langle p/2\rangle},A^{1/2}f^{\langle p/2\rangle}\rangle$, see \cite[Theorem~1]{MR1160303}, see also Kinzebulatov and Semenov \cite[Proposition 8]{kinzebulatov2020fractional} 
for recent developments 
and \cite{kinzebulatov2020admissible} for applications to
symmetric stable processes with drift, understood as 
solutions of stochastic differential equations.

Finally, it may be hard to point out the first occurrence of \eqref{e.cHk}. Our best guess is \cite[Lemma 1]{MR1160303} and  \cite[Lemma 2.1]{MR1407327}, see also  \cite[p. 246]{MR803094}, Bakry \cite[p. 39]{MR1307412}, Stroock \cite[Lemma 9.9 and p. 134]{MR755154} and Carlen, Kusuoka and Stroock \cite[p. 269]{MR898496} for formulations with nonnegative arguments or one-sided comparison. We also point out to the calculations with forms and powers in Davies \cite[Chapter 2 and 3]{MR1103113}, to the calculations in
\cite[Section 7.6]{MR3235527} and
\cite[(2.19)]{2020arXiv200601932B}, and to Lemma~\ref{thm:incrineq} below.

\section{{Preliminaries}}\label{sec:Pre}

We use ``$:=$" to indicate definitions, e.g.,
$a \wedge b := \min \{ a, b\}$, $a \vee b := \max \{ a, b\}$, and $a_+:=a\vee 0$.
All the functions considered below are Borel measurable either by construction or assumptions. If $\mathcal F$ is a family of functions, then we let $\mathcal F_+=\{f\in \mathcal F: f\ge 0\}$.
For nonnegative functions $f$ and $g$ we write $f(x)\approx g(x)$ to indicate that there are numbers $0<c<C<\infty$ such that $cf(x)\le  g(x)\le C f(x)$ for all the considered arguments $x$.
We call such comparisons two-sided or \textit{sharp}.
For an open subset $D$ of the $d$-dimensional Euclidean space $\R^d$, we let $C^\infty_c (D)$ be the space of smooth functions with compact support in $D$.

Note that $(|x|^p)'=p x^{\langle p - 1 \rangle}$ on $\R$ for $p\in (1,\infty)$. We shall focus on related inequalities.
\begin{lem}
\label{thm:incrineq}
For $p \in (1,\infty)$ there are constants $C,C'>0$ such that for $a,b\in \R$,
\begin{align}\label{eq:2d}
0\leq|b|^p - |a|^p - pa^{\langle p - 1 \rangle} (b-a) &\leq C |b-a|^{\lambda} (|b|+|a|)^{p-\lambda}, \quad \lambda \in [0,2],\\
\label{eq:2c}
||b|^p - |a|^p|&\le (p+C) |b-a|(|b|+|a|)^{p-1},\\
\label{eq:2dd}
| b^{\langle p - 1 \rangle} - a^{\langle p - 1 \rangle}|&\leq C' |b-a|^{\lambda} (|b|+|a|)^{p-1-\lambda}, \quad \lambda \in [0,1].
\end{align}
\end{lem}
\begin{proof}
The inequality \eqref{eq:2d} with $\lambda=2$ follows from \cite[Lemma 6]{MR3251822}. In fact, we have
\begin{equation}\label{eq:2e}
|b|^p - |a|^p - pa^{\langle p - 1 \rangle} (b-a) \approx (b-a)^2 (|b|+|a|)^{p-2}, \quad a,b\in \R,
\end{equation}
see \cite[(11)]{MR3251822}.
Here is, however, another proof:
if $a \neq 0$, then we let $x=b/a$ and arrive at the following equivalent statement of \eqref{eq:2d}:
$$
0\le |x|^p-1 - p(x-1) \leq C |x-1|^\lambda (|x|+1)^{p-\lambda}, \quad x \in \R.
$$
Since $x\mapsto |x|^p$ is strictly convex, both sides are continuous and strictly positive on $\R\setminus\{1\}$. Let first $\lambda=2.$
By a compactness argument it is enough to notice that
$$
\lim_{x \to \pm \infty} \frac{|x|^p  -1-p(x-1)}{(x-1)^2(|x|+1)^{p-2}} = 1,
$$
and, by L'H\^{o}pital's rule,
\begin{equation*}
\lim_{x \to 1} \frac{|x|^p  -1 -p(x-1)}{(x-1)^2}
= \frac{p(p-1)}{2}.
\end{equation*}
 When $\lambda\in [0,2]$, we have that  $|b-a|^2(|b|+|a|)^{p-2}\leq |b-a|^\lambda (|b|+|a|)^{p-\lambda}$, which gives
 \eqref{eq:2d}.
\eqref{eq:2c} follows from   \eqref{eq:2d} with $\lambda=1$, and
\eqref{eq:2dd} can be proved like \eqref{eq:2d}.
\end{proof}

\subsection{Functions and kernels}
Let $g$ denote the Gaussian kernel
\begin{equation}\label{e:Gk}
g_t(x) = (4\pi t)^{-d/2} e^{-|x|^2/(4t)}\,, \quad t>0,\quad x\in \Rd\,.
\end{equation}
Let $\eta_t(s)\ge 0$ be the density function of the $\alpha/2$-stable subordinator at time~$t$, see, e.g., Bliedtner and Hansen \cite[Section V.3]{MR850715}, Schilling, Song and Vondra{\v{c}}ek \cite[Chapter 13]{MR2978140} or Sato \cite[Section 6.30]{MR1739520}.
In particular,
$\eta_t(s) = 0$ for $s\le 0$, and we have the following formula for the Laplace transform of $\eta_t$:
\[
 \int_0^\infty e^{-us} \eta_t(s)\,ds = e^{-tu^{\alpha/2}}, \quad u\geq 0, \quad t>0.
\]
We define, by \textit{subordination}, the convolution semigroup of functions (the $\alpha$-stable semigroup):
\begin{equation}\label{e:stable-pt}
 p_t(x) = \int_0^\infty g_s(x) \eta_t(s)\,ds \,, \quad t>0,\quad x\in \Rd.
\end{equation}
It is well known that the functions $p_t$ satisfy the following scaling property:
\begin{equation}
  \label{e:sca}
  p_t(x)=t^{-\frac{d}{\alpha}}p_1(t^{-\frac{1}{\alpha}}x)\, \quad t>0,\quad x\in \Rd\,.
\end{equation}
We consider the following transition probability density
\begin{equation}\label{e.dp}
p_t(x,y) = p_t(y-x), \quad t>0, \, x,y \in \Rd.
\end{equation}
Clearly, $p_t(x,y)$ is symmetric in the space variables: $p_t(x,y)=p_t(y,x)$.
It is well known that
\begin{equation}\label{e.mpa}
p_t(x,y) \leq c \min \left(t^{-d/\alpha}, t|x-y|^{-d-\alpha}\right), \quad t>0,\quad x, y \in \Rd,
\end{equation}
and so we can record the following convenient estimate
\begin{equation}\label{e.mp}
p_t(x,y)/t \leq c \nu(x,y), \quad t>0,\quad x, y \in \Rd,
\end{equation}
see, e.g., Bogdan, Grzywny and Ryznar \cite[Theorem 21]{MR3165234}. Also,
\begin{equation}\label{e.zp}
p_t(x,y)/t \rightarrow \nu(x,y) \mbox{ as } t \rightarrow 0^+,
\end{equation}
see, e.g., Cygan, Grzywny and Trojan \cite[Proof of Theorem 6]{MR3646773}. We will use \eqref{e.mp} and \eqref{e.zp} to support arguments based on the following general fact.
 \begin{lem}\label{l:ru}
If nonnegative functions satisfy
$f_n\le c f$ for $n\in \N$, and $f=\lim f_n$, then
$\lim \int f_n d\mu= \int f d\mu$ for any measure $\mu$.
\end{lem}
\begin{proof}
If the integral $\int f d\mu$ is finite, then the dominated convergence theorem applies. Otherwise, by Fatou's lemma, $ \int f d\mu=\infty=\liminf \int f_n d\mu=\lim \int f_n d\mu$, too.
\end{proof}

For $\alpha < d$ and $\beta \in (0, d)$, we let
$$f_{\beta}(t) = c t^{(d-\alpha -\beta)/\alpha}_+ ,\quad t\in \R.$$
Here $c\in (0, \infty)$ is a normalizing constant so chosen that
\begin{equation}\label{e.dfb}
 \int_0^\infty f_{\beta}(t) p_t(x) dt =|x|^{-\beta}=h_\beta(x),  \quad x\in \Rd.
\end{equation}
For $\beta \in (0, d-\alpha)$ we let
\begin{equation}\label{e.dqb}
     q_\beta(x): = \frac{1}{h_\beta(x)}\int_0^{\infty }f_{\beta}'(t) p_t(x) dt,\quad x\in \Rd.
\end{equation}
We note that the choice of $c$ does not affect $q_\beta$.
By \cite[Section~4]{MR3460023},
\begin{equation}\label{e:qbeta}
  q_\beta(x) = \kappa_\beta |x|^{-\alpha }.
\end{equation}
By \cite[the proof of Proposition~5]{MR3460023}, the function $\beta \mapsto \kappa_\beta$ is increasing on $(0,(d-\alpha)/2]$ and decreasing on $[(d-\alpha)/2,d-\alpha)$. Furthermore, $\kappa_\beta = \kappa_{d-\alpha-\beta}$.

We denote, as usual, $P_t u(x)=\int_{\Rd} u(y) p_t(x,y) \, dy$ for $u:\Rd \to \R$ and $x \in \Rd$, if the integral is well defined.

Since $\int_\Rd p_t(x,y)dy=\int_\Rd p_t(x,y)dx=1$, by Schur's test, for every $p\in [1,\infty]$,
\begin{equation}\label{e.cLp}
\|P_tf\|_p\le \|f\|_p, \quad f\in L^p(\Rd).
\end{equation}
In fact, since $p_t(x,y)\le p_t(0,0)=t^{-d/\alpha}p_1(0,0)<\infty$, by Young inequality,
\begin{equation}\label{e.cLi}
\|P_tf\|_\infty \le c_t \|f\|_p, \quad f\in L^p(\Rd).
\end{equation}
By \cite[Eq. 7]{MR3460023}, for $0\leq \beta \leq d-\alpha$,
\begin{equation}\label{e:supermedian}
P_th_\beta \leq h_\beta.
\end{equation}
In this sense, $h_\beta$ is supermedian when $0\leq \beta \leq d-\alpha$.

\subsection{The domain of the form $\mathcal E_p$}
Recall that $p \in (1,\infty)$. As usual, $L^p(\Rd)$ is the collection of all  the real-valued functions on $\Rd$ such that $\int_{\Rd} |f(x)|^p dx < \infty$;  we identify functions $u,v \in L^p(\Rd)$ if $u=v$ a.e. on $\Rd$.  The dual space of $L^p(\Rd)$ is, of course, $L^q(\Rd)$, where $q=p/(p-1)$.  We also consider the canonical pairing $\langle u, v\rangle= \langle v, u \rangle = \int_{\Rd} u(x)v(x)dx$, $u \in L^p(\Rd), v \in L^q(\Rd)$.

Since for any $a,b\in \R$ we have $(b-a)(b^{\langle p-1\rangle}-a^{\langle p-1 \rangle})\geq 0$, the nonlinear form
\[\EEE_p{[u]} = \frac{1}{2} \int_{\mathbb R^d}\int_{\mathbb R^d} (u(x)-u(y))(u^{\langle p-1\rangle}(x)-u^{\langle p-1 \rangle}(y))\nu(x,y) \,dy \,dx\] is well-defined, perhaps infinite, for every $u:\Rd\to\R.$ We define its natural domain:
\begin{equation}\label{def:DE}
\mathcal D(\mathcal E_p)= \{u\in L^p(\Rd) : \mathcal E_p{[u]}<\infty\}.
\end{equation}
For $p=2$, as usual, we let
\begin{equation*}
\mathcal E{[u]}=\lim_{t\to 0}\frac{1}{t}\langle u-P_tu, u\rangle,
\end{equation*}
where $u\in L^2(\Rd)$ is arbitrary and the expression under the limit is decreasing in $t$ (see Hille and Phillips \cite[Section 22.3]{MR0423094} and \cite[Lemma 1.3.4]{MR2778606}).
We also let
\begin{equation*}
\mathcal D(\mathcal E)= \{u\in L^2(\Rd): {\EEE[u]<\infty}\}.
\end{equation*}
Note that the quantity  suggested by the definition of
$\EEE[u]$, \[\EEE^{(t)}(u,v): = \frac{1}{t} \langle u - P_t u, v \rangle, \quad t>0,\]
makes sense for all
 $u \in L^p(\Rd)$ and $v \in L^q(\Rd)$ because $P_tL^p(\Rd)\subset L^p(\Rd)$. In fact, $(P_t)_{t\ge 0}$ form a strongly continuous semigroup of contractions on $L^p(\Rd)$, see, e.g., Kwa\'{s}nicki \cite{MR3613319}, and we denote by $\mathcal D_p(\Delta^{\alpha/2})$ the domain of $\Delta^{\alpha/2}$ considered as the generator of this semigroup.
We can now characterize $\mathcal D(\EEE_p)$. We note that \eqref{e.cHkf} is  a variant of  \cite[Theorem 3.1]{MR1407327}, see also \cite[Remark (3)]{MR1407327}.
\begin{lem}
Let $p>1$. Then for every $u\in L^p(\mathbb R^d)$ we have
\begin{equation}\label{e:wpf}
\EEE_p[u]= \lim_{t\to 0} \EEE^{(t)}(u,u^{\langle p-1 \rangle}).
\end{equation}
Furthermore,
\begin{align}\label{e:dom_def-1}
 \mathcal D(\EEE_p) & = \{u\in L^p(\Rd): \sup_{t>0}\mathcal E^{(t)}(u,u^{\langle p-1 \rangle})<\infty\}\\
 \label{e:dom_def-2}
 &=\{u\in L^p(\Rd): \text{ finite } \lim_{t\to 0}\mathcal E^{(t)}(u,u^{\langle p-1 \rangle}) \text{ exists}\}.
 \end{align}
For arbitrary Borel measurable $u: \Rd \to \R$ we also have
\begin{equation}\label{e.cHkf}
\frac{4(p-1)}{p^2}\EEE[u^{\langle p/2\rangle}]\le \EEE_p[u]\le 2 \EEE[u^{\langle p/2\rangle}]
\end{equation}
and $\mathcal D(\EEE_p)=\mathcal D(\EEE)^{\langle 2/p\rangle}:=\{v^{\langle 2/p\rangle}: v\in \mathcal D(\EEE)\}$.
Finally, $\mathcal D_p(\Delta^{\alpha/2})\subset  \mathcal D(\EEE_p)$ and
\begin{equation}\label{e.pfag}
\EEE_p[u]=-\langle\Delta^{\alpha/2}u, u^{\langle p-1\rangle}\rangle, \quad u\in \mathcal D_p(\Delta^{\alpha/2}).
\end{equation}
\end{lem}

\begin{proof}
Since $\int_{\Rd} p_t(x,y) dy = 1$ for all $t>0, \ x \in \Rd$, the symmetry of $p_t(x,y)$
yields
\begin{align*}
\EEE^{(t)}(u, u^{\langle p - 1 \rangle} ) &= \frac{1}{t} \int_{\Rd} \int_{\Rd} p_t(x,y) (u(x)-u(y)) dy \, u(x)^{\langle p - 1 \rangle}  \, dx \\
&= \frac{1}{t} \int_{\Rd} \int_{\Rd} p_t(y,x) (u(y)-u(x)) dx \, u(y)^{\langle p - 1 \rangle}  \,dy \\
&= \frac{1}{2t} \int_{\Rd} \int_{\Rd} p_t(x,y) (u(x)-u(y)) (u(x)^{\langle p - 1 \rangle}  - u(y)^{\langle p - 1 \rangle} ) \,dy \,dx
\end{align*}
for every $u\in L^p(\mathbb R^d)$.
By Lemma~\ref{l:ru}, \eqref{e.mp} and \eqref{e.zp}
we get \eqref{e:wpf}. Since the limit  in \eqref{e:wpf} exists for all $u\in L^p(\mathbb R^d)$, we obtain \eqref{e:dom_def-1} and \eqref{e:dom_def-2}, using dominated convergence.
From \eqref{e.cHk} we obtain \eqref{e.cHkf}. Of course, if $v\in \mathcal D(\EEE)$ and $u=v^{\langle 2/p\rangle}$, then by \eqref{e.cHkf}, $\EEE_p[u]\le 2 \EEE[(v^{\langle 2/p\rangle})^{\langle p/2\rangle}]=2\EEE[v]<\infty$, so $u\in \mathcal D(\EEE_p)$. Conversely, if $u\in \mathcal D(\EEE_p)$ and we define $v=u^{\langle p/2\rangle}$, then by \eqref{e.cHkf}, $\EEE[v]\le p^2(4(p-1))^{-1} \EEE_p[u]<\infty$, so $v\in \mathcal D(\EEE)$ -- and clearly $u=v^{\langle 2/p\rangle}$.

Let $u \in \mathcal D_p(\Delta^{\alpha/2})$. Consider again $\frac{1}{t} \langle u - P_t u,u^{\langle p-1\rangle} \rangle$ as $t\to 0^+$. By the definition of the semigroup generator \cite{MR3613319}, we see that
$(P_tu-u)/t$ converges to $\Delta^{\alpha/2}u$ in $L^p(\Rd)$. Since $u^{\langle p-1\rangle}\in L^{p/(p-1)}(\Rd)$, the expression $\mathcal E^{(t)}(u,u^{\langle p-1\rangle})$ tends to (finite) $-\langle\Delta^{\alpha/2}u, u^{\langle p-1\rangle}\rangle$. On the other hand, it converges to $\EEE_p[u]$ by \eqref{e:wpf}.
\end{proof}

\section{Hardy identity and inequality}
\label{sec:Hi}
\begin{proof}[Proof of Theorem~\ref{thm:Theorem1}] For $\beta$ as in the assumptions of the theorem, we write $h=h_\beta$, $f=f_\beta$, etc. Take $u\in L^p(\Rd)$ and
let $v=u/h$, so that $vh=u$ a.e. The factorization $vh=u$ is the essence of Doob's conditioning, which inspired the calculations in \cite[Theorem 2]{MR3460023} and in what follows. Let $t>0$. Of course,
$vh \in L^p(\Rd)$, and by \eqref{e:supermedian}, $vP_th \in L^p(\Rd)$. Consider \eqref{e:wpf}. We have
\begin{align*}
\EEE^{(t)}(vh, (vh)^{\langle p - 1 \rangle} ) &= \frac{1}{t} \langle vh, (vh)^{\langle p - 1 \rangle}  \rangle
 - \frac{1}{t}\langle P_t(vh), (vh)^{\langle p - 1 \rangle}  \rangle
\\
& =  \frac{p-1}{p} \langle v \frac{h-P_th}{t}, (vh)^{\langle p - 1 \rangle}  \rangle  + \frac{p-1}{p} \langle v\frac{P_th}{t} , (vh)^{\langle p - 1 \rangle}  \rangle\\
&\quad+ \frac{1}{p} \langle v^{\langle p-1\rangle} \frac{h^{p-1}-P_th^{p-1}}{t}, vh \rangle+ \frac{1}{p} \langle v^{\langle p - 1 \rangle} \frac{ P_th^{p-1}}{t}, vh\rangle \\
&\quad- \langle \frac{ P_t(vh)}{t}, (vh)^{\langle p - 1 \rangle}  \rangle = I_t^{(1)} + I_t^{(2)} + J_t,
\end{align*}
where
\begin{align*}
I_t^{(1)}&:=  \frac{p-1}{p} \langle v \frac{h-P_th}{t}, (vh)^{\langle p - 1 \rangle}  \rangle=\frac{p-1}{p} \int_{\Rd}  |v(x)|^p h(x)^{p-1} \frac{(h-P_th)(x)}{t}\,dx,\\
I_t^{(2)}&:=  \frac{1}{p} \langle v^{\langle p-1\rangle} \frac{h^{p-1}-P_th^{p-1}}{t}, vh \rangle
=\frac{1}{p} \int_{\Rd}  |v(x)|^p h(x) \frac{(h^{p-1}-P_th^{p-1})(x)}{t}\,dx, \\
J_t&:=\frac{1}{p} \langle v^{\langle p - 1 \rangle} \frac{ P_th^{p-1}}{t}, vh \rangle + \frac{p-1}{p} \langle v\frac{P_th}{t} , (vh)^{\langle p - 1 \rangle}  \rangle - \langle \frac{ P_t(vh)}{t}, (vh)^{\langle p - 1 \rangle}  \rangle.
\end{align*}
By the symmetry of $p_t(x,y)$ and the definition of $F_p$,
\begin{align*}
J_t &= \frac{1}{p} \int_{\Rd} \int_{\Rd} |v(x)|^p h(y)^{p-1}h(x) \, \frac{p_t(x,y)}{t} \,dy \,dx \\
&\quad + \frac{p-1}{p} \int_{\Rd} \int_{\Rd} |v(x)|^p h(x)^{p-1} h(y)  \, \frac{p_t(x,y)}{t} \,dy \,dx \\
&\quad - \int_{\Rd} \int_{\Rd} v(x)^{\langle p - 1 \rangle}  v(y) h(x)^{p-1} h(y)  \, \frac{p_t(x,y)}{t} \,dy \,dx \\
&= \frac{1}{p}\int_{\Rd} \int_{\Rd} \left[|v(y)|^p - |v(x)|^p -p v(x)^{\langle p - 1 \rangle}  \left(v(y) - v(x) \right) \right] h(x)^{p-1}h(y) \, \frac{p_t(x,y)}{t} \, dy \,dx \\
&=\frac{1}{p}\int_{\Rd} \int_{\Rd} F_p\left({v(x)}, {v(y)} \right) h(x)^{p-1}h(y) \frac{p_t(x,y)}{t} \, dy \, dx.
\end{align*}
By Lemma~\ref{l:ru}, \eqref{e.mp} and \eqref{e.zp},
\begin{align}\label{e.f1}
\lim_{t\to 0}J_t &=\frac{1}{p}\int_{\Rd} \int_{\Rd} F_p\left({v(x)}, {v(y)} \right) h(x)^{p-1}h(y) \nu(x,y) \, dy \, dx \nonumber \\
&=\frac{1}{p}\int_{\Rd} \int_{\Rd} F_p\left(\frac{u(x)}{h(x)}, \frac{u(y)}{h(y)} \right) h(x)^{p-1}h(y) \nu(x,y) \, dy \, dx.
\end{align}
We next consider $I_t^{(1)}$ and
recall that $h(x)=\int_0^\infty f(s) p_s(x)\,ds$, so for $x\neq 0$,
\begin{align*}
 (h-P_th)(x) &= \int_0^\infty f(s) p_s(x) \,ds - \int_0^\infty f(s) p_{s+t}(x) \,ds \\
&=\int_0^\infty [f(s) - f(s-t)] p_s(x) \,ds.
\end{align*}
Thus,
\begin{align*}
I_t^{(1)}&=\frac{p-1}{p} \int_{\Rd} |v(x)|^p h(x)^{p-1} \int_0^\infty \frac{1}{t}[f(s) - f(s-t)]\ p_s(x) \,ds\, dx.
\end{align*}
We have $0\leq f(s)-f(s-t)\leq C t f'(s)$ for $s,t>0$.
By Lemma~\ref{l:ru} and \eqref{e:qbeta} we get
\begin{align}
\underset{t \to 0} {\lim} \, I_t^{(1)}&= \frac{p-1}{p}  \int_{\Rd}  \, |v(x)|^p h(x)^{p-1} \int_0^{\infty} f'(s)p_s(x)\, ds\, dx\nonumber\\
&= \frac{p-1}{p}  \int_{\Rd} |v(x)|^p h(x)^{p}q(x) \, dx \label{e.f2}
=
\frac{(p-1)\kappa_{\beta}}{p} \int_{\Rd} \frac{|u(x)|^p}{|x|^{\alpha}}\, dx.
\end{align}
We can treat $I_t^{(2)}$ analogously. The analogy has several layers. We let $p'=p/(p-1)$. We have $p=p'/(p'-1)$ and $(p'-1)(p-1)=1$.
By considering
\eqref{e:hbeta} and \eqref{e.dfb} we get
\begin{align*}
I_t^{(2)} &=\frac{p'-1}{p'} \int_{\Rd}  |v(x)^{\langle 1/(p'-1) \rangle}|^{p'} h^{p'-1}_{(p-1)\beta}(x) \frac{h_{(p-1)\beta}(x)-P_th_{(p-1)\beta}(x)}{t} \, dx.
\end{align*}
We note that $0\le (p-1)\beta \leq d-\alpha$ and by the case of $I^{(1)}_t$ we get
\begin{align}\label{e.f3}
\underset{t \to 0} {\lim} \,I_t^{(2)}&=\frac{\kappa_{(p-1)\beta}}{p} \int_{\Rd} \frac{|u(x)|^p}{|x|^{\alpha}}\, dx.
\end{align}
By \eqref{e:wpf}, \eqref{e.f1}, \eqref{e.f2} and \eqref{e.f3} we get
\eqref{e:main}.
\end{proof}
For $\beta = (d-\alpha)/p$, we have $\kappa_{(p-1)\beta} = \kappa_{d-\alpha-\beta} = \kappa_{\beta}$, which yields \eqref{e:ip}. As aforementioned in Introduction, our constant is better than that in \eqref{e:subgoal}.
\begin{lem}\label{thm:lem7}
If $p\neq 2$, then
$
\kappa_{(d-\alpha)/p} > 4(p-1)p^{-2}\kappa_{(d-\alpha)/2}
$.
\end{lem}
\begin{proof}
Let $p>2$.
Denote
$$
r(t)=\kappa_{2t}=\frac{{2^\alpha}\Gamma(\frac{\alpha}{2}+t) \Gamma(\frac{d}{2}-t)}{\Gamma(t)\Gamma((d-\alpha)/2 -t)}, \quad t\in(0,\tfrac{d-\alpha}{4}).
$$
Put $t=\frac{d-\alpha}{2p}$. Then $p=\frac{d-\alpha}{2t}$ and
$$
\frac{4(p-1)}{p^2} = \frac{16t^2}{(d-\alpha)^2} \left(\frac{d-\alpha}{2t}-1\right) = \frac{16}{(d-\alpha)^2} t \left((d-\alpha)/2 -t \right) =: s(t).
$$
We only need to verify that
$$
r(t) > \kappa_{(d-\alpha)/2} s(t), \quad t \in (0,\tfrac{d-\alpha}{4}).
$$
Notice that $r(0^+) =0$ and  $s(0^+)=0$ and $r(\frac{d-\alpha}{4})=\kappa_{(d-\alpha)/2} = \kappa_{(d-\alpha)/2} s(\frac{d-\alpha}{4})$.
Let
$$
F(t) = \ln r(t) - \ln \kappa_{(d-\alpha)/2} s(t).
$$
It suffices to show that $F(t)>0$ for $t \in (0,\frac{d-\alpha}{4})$.
By \cite[Proof of Proposition 5]{MR3460023},
\begin{align*}
F'(t) &= \sum_{k=0}^{\infty} \left(\frac{1}{\frac{d}{2}+k-t} - \frac{1}{t+\frac{\alpha}{2}+k} - \frac{1}{(d-\alpha)/2+k-t} + \frac{1}{t+k}\right) - \left(\frac{1}{t} - \frac{1}{(d-\alpha)/2 -t} \right),
\end{align*}
hence
\begin{align*}
F'(t) &= \sum_{k=0}^{\infty} \left(\frac{1}{\frac{d}{2}+k-t} - \frac{1}{t+\frac{\alpha}{2}+k} - \frac{1}{(d-\alpha)/2+1+k-t} + \frac{1}{t+1+k}\right)\\
&= \sum_{k=0}^{\infty} \left(\frac{\frac{\alpha}{2}-1}{(t+1+k)(t+\frac{\alpha}{2}+k)} - \frac{\frac{\alpha}{2}-1}{(\frac{d}{2}+k-t)(\frac{d-\alpha+2}{2}+k-t)} \right) \\
&= \frac{\alpha-2}{2} \sum_{k=0}^{\infty} \frac{k(d-\alpha)+(d-\alpha)(d+2)/4 - t(d+2+4k)}{(t+1+k)(t+\frac{\alpha}{2} +k)(\frac{d}{2}+k-t) (\frac{d-\alpha+2}{2}+k-t)}.
\end{align*}
The numerator of every term in the last series is decreasing in $t$ and equal to $0$ for $t=\frac{d-\alpha}{4}$, so 
$F'(t) <0$ for $t\in(0,\frac{d-\alpha}{4})$. Since $F(\frac{d-\alpha}{4})=0$,  $F(t)>0$ for $t\in(0,\frac{d-\alpha}{4})$.
We now consider the case $p \in(1,2)$. Let $q= \frac{p}{p-1}$. Of course, $q \in (2, \infty)$ and $\frac{1}{p}+\frac{1}{q}=1$.
We have
$ \kappa_{(d-\alpha)/q}=  \kappa_{(d-\alpha)/p}$, since $(d-\alpha)/q+(d-\alpha)/p=d-\alpha$ and
$\kappa_\beta$ is symmetric with respect to $(d-\alpha)/2$.
Furthermore, $4(p-1)p^{-2}=4/(pq)=4(q-1)q^{-2}$.
\end{proof}

\section{Optimality}\label{s:opt}

\noindent
In this section we prove Theorem~\ref{thm:Theorem2}.
The argument is rather technical due to integrability issues with the intended test function
for \eqref{e.Hp}.
We start with three auxiliary lemmas.
The following decomposition of $\mathcal{E}_p{[u]}$ is different and simpler than \eqref{e:main} but we should note that the second term needs not be nonnegative or finite.
\begin{lem}\label{thm:lem4}
Under the assumptions of Theorem~\ref{thm:Theorem1} we have
\begin{align*}
&\mathcal{E}_p{[u]} = \kappa_{\beta} \int_{\Rd} \frac{\left|u(x)\right|^p}{\left|x\right|^{\alpha}} dx \\
&+ \underset{t\to 0}\lim \ \frac{1}{2}\int_{\Rd} \int_{\Rd} \left( \frac{u(x)}{h(x)}-\frac{u(y)}{h(y)} \right) \left( \frac{u(x)^{\langle p - 1 \rangle} }{h(x)}-\frac{u(y)^{\langle p - 1 \rangle} }{h(y)} \right) h(x) h(y) \frac{p_t(x,y)}{t} \,dy \,dx,
\end{align*}
if additionally $u \in L^p(\Rd, |x|^{-\alpha} dx)$.
\end{lem}
\begin{proof}
Thus, $u \in L^p(\Rd, (1+ |x|^{-\alpha}) dx)$. Let $v=u/h$. By \eqref{e:wpf},
$$\mathcal{E}_p{[u]} = \lim_{t\to 0^+} \EEE^{(t)}(vh,(vh)^{\langle p - 1 \rangle} ).$$
By the proof of Theorem~\ref{thm:Theorem1}, $vP_th \in L^p(\Rd)$. We have
\begin{align*}
\EEE^{(t)}(vh, (vh)^{\langle p - 1 \rangle} ) &= \langle v \frac{h-P_th}{t}, (vh)^{\langle p - 1 \rangle}  \rangle +
   \langle  \frac{vP_th-P_t(vh)}{t} , (vh)^{\langle p - 1 \rangle}  \rangle =: I_t + J_t.
\end{align*}
By \eqref{e.f2} and the assumption $u \in L^p(\Rd, |x|^{-\alpha} dx)$
\begin{align*}
\underset{t \to 0}{\lim} \, I_t = \kappa_{\beta} \int_{\Rd} \frac{|v(x)h(x)|^p}{|x|^{\alpha}} \,dx<\infty.
\end{align*}
Symmetrizing the integrand in $J_t$ we get
\begin{align*}
& J_t  =
 \int_{\Rd} \int_{\Rd} \left(v(x)-v(y)\right) h(y) v(x)^{\langle p - 1 \rangle}  h(x)^{p-1} \, \frac{p_t(x,y)}{t}\, dy \, dx \\
&=\int_{\Rd} \int_{\Rd} \!\!\!\left(v(x)-v(y)\right) \left( v(x)^{\langle p - 1 \rangle}  h(x)^{p-1}h(y) - v(y)^{\langle p - 1 \rangle} h(y)^{p-1}h(x) \right) \frac{p_t(x,y)}{2t} dy dx.
\end{align*}
Since $u=vh$, the result follows. \end{proof}

Here is an informal idea of the proof of Theorem~\ref{thm:Theorem2}.
Let
\begin{equation}\label{e.mi2}
u(x)=\left|x \right|^{-\beta/(p-1)} \wedge \left|x \right|^{-\beta}, \quad x\in \Rd.
\end{equation}
Since $h=h_\beta$, the integrand under the limit in Lemma~\ref{thm:lem4} is zero if $|x|,|y|\le 1$ and if $|x|,|y|\ge 1$, and it has negative sign otherwise.
This reverses the Hardy inequality, but, unfortunately we need to face
the aforementioned integrability issues.
The next result gives a preparation.
For $r>0$ we let $B(0,r) = \{x \in \Rd: |x|<r\}$. Also, we let $B_1=B(0,1)$ and $B_1^c=\Rd\setminus B_1$.
\begin{lem}\label{thm:lem5}
If $p>2$ and $\mu>\beta\vee (d-\alpha)/p$, then
\begin{align*}
\int_{B_1^c}\int_{B_1^c}\left| |x|^{\beta-\mu} - |y|^{\beta-\mu} \right| \left||x|^{\beta-(p-1)\mu}-|y|^{\beta-(p-1)\mu}\right|&|x|^{-\beta}|y|^{-\beta} \nu(x,y) \, dy \, dx
<\infty.
\end{align*}
\end{lem}
\begin{proof}
For $x \in B_1^c$ we let
$D_1(x) = B_1^c \cap B(0, |x|/2)=\{y:1\leq |y|< |x|/2\}$, $D_2(x) = B_1^c \cap B(0, 2|x|)^c = B(0,2|x|)^c=\{y:2|x|\leq |y|\}$, and $D_3(x) = B_1^c \setminus (D_1(x) \cup D_2(x))=\{y:|x/2|\vee 1\leq |y|<2|x|\}$. The sets form a partition of $B_1^c$.
Note that $p\mu+\alpha > (p\beta)\vee(d-\alpha)+\alpha\ge d$.
With the above integrand, $\int_{B_1^c}\int_{D_1(x)}=\int_{B_1^c}\int_{D_2(x)}$.
Indeed, the integrand is symmetric in $x,y$, on the left we integrate over $\{x,y\in B_1^c: 2|y|<|x|\}$, and on the right -- over $\{x,y\in B_1^c: 2|x|<|y|\}$.
 Then for $y \in D_1(x)$ we have $|y| \le |x|$ and $|x-y| \ge |x|/2$, hence
\begin{align*}
\int_{B_1^c}\int_{D_1(x)}
&\le 2^{d+\alpha} \int_{B_1^c} \int_{D_1(x)} |x|^{-d-\alpha} |y|^{-p\mu}  dy \, dx\\
&= 2^{d+\alpha}\int_{B_1^c}\int_{B(0,2|y|)^c} |x|^{-d-\alpha} |y|^{-p\mu}   dx \, dy
= c  \int_{B_1^c} |y|^{-p\mu-\alpha} \, dy<\infty.
\end{align*}
For $y \in D_2(x)$ we have $|y| \ge |x|$ and $|x-y| \ge |y|/2$, hence
\begin{align*}
\int_{B_1^c}\int_{D_2(x)}
&\le 2^{d+\alpha} \int_{B_1^c} \int_{D_2(x)} |x|^{\beta-p\mu} |y|^{-\beta-d-\alpha} dy \, dx = c \int_{B_1^c} |x|^{-p\mu-\alpha} dx<\infty.
\end{align*}
We next note that for $a,b>0$ and $\gamma>0$,
\begin{equation}\label{eq:lemlem}
|a^{-\gamma} -b^{-\gamma}| \leq \gamma |a-b| (a\wedge b)^{-\gamma-1}.
\end{equation}
Indeed, if, say, $a\leq b$, then
$$
a^{-\gamma} - b^{-\gamma} = \gamma \int_a^b s^{-\gamma-1} ds \leq \gamma (b-a)a^{-\gamma-1} .
$$
For $y \in D_3(x)$ we have $|x|/2 <|y|<2|x|$, therefore,
\begin{align*}
\int_{B_1^c} \int_{D_3(x)}
&\leq  c_1 \int_{B_1^c} \int_{D_3(x)} |x|^{-p\mu -2} |x-y|^{-d-\alpha+2} dy \, dx \\
&\leq  c_{1} \int_{B_1^c} \int_{B(0,3|x|)} |x|^{-p\mu -2} |z|^{-d-\alpha+2} dz \, dx
= c_{2} \int_{B_1^c} |x|^{-p\mu-\alpha} dx<\infty.
\end{align*}
\end{proof}
Here is another auxiliary estimate.
\begin{lem}\label{thm:lem6}
If $p>2$, $0<\beta<\mu$, $R>1$ and $A_R = \{x \in \Rd: 1/R<|x|<1\}$, then
\begin{align}
\nonumber
&\int_{A_R}\int_{B_1^c} \left| 1-|y|^{\beta-\mu}\right| \left||x|^{-(p-2)\beta}-|y|^{\beta-(p-1)\mu} \right||x|^{-\beta}|y|^{-\beta} \nu(x,y) \, dy \,dx \\
& \le c (1 \vee R^{(p-1)\beta-d} \vee \operatorname{log} R).
\end{align}
\end{lem}
\begin{proof}
We split the integral as follows:
$$
\int_{A_R}\int_{B_1^c}=\int_{D_1}\int_{B_1^c} + \int_{D_2}\int_{D_3} + \int_{D_2}\int_{D_4} =: I_1 + I_2 + I_3,
$$
where $D_1 = \{ 1/R < |x| \leq 1/2 \}$, $D_2 = \{ 1/2 \le |x|<1\}$, $D_3 = \{1<|y| < 2\}$ and $D_4 = \{|y| \geq 2\}$.
If $x \in A_R$ and $y \in B_1^c$, then $|y|^{\beta-\mu} \le 1$ and $|x|^{-(p-2)\beta} \ge |y|^{\beta-(p-1)\mu}$.
Furthermore, for $x \in D_1$ and $y \in B_1^c$, we have $|x-y|>|y|/2$, hence
\begin{align*}
I_1 &\leq c_1  \int_{D_1} \int_{B_1^c} |x|^{-(p-1)\beta} |y|^{-d-\alpha-\beta} dy \, dx
=c_2 \int_{D_1} |x|^{-(p-1)\beta} dx \\
& \le c(1 \vee R^{(p-1)\beta  - d} \vee \operatorname{log} R).
\end{align*}
If $x \in D_2$, $y\in D_3$, then $1-|y|^{\beta-\mu} < |x|^{\beta-\mu} - |y|^{\beta-\mu}$ and $|x|^{-(p-2)\beta}-|y|^{\beta-(p-1)\mu} < |x|^{\beta-(p-1)\mu}-|y|^{\beta-(p-1)\mu}$, so by \eqref{eq:lemlem},
\begin{align*}
|I_2| &\leq c \int_{D_2} \int_{D_3} |x|^{\beta-p\mu-2} |y|^{-\beta} |x-y|^{-d-\alpha+2} dy \, dx \leq c_1 \int_{D_2}\int_{D_3} |x-y|^{-d-\alpha+2} dy \, dx \\
&\leq c_1 \int_{D_2}\int_{B(0,3)} |z|^{-d-\alpha+2} dz \, dx < \infty.
\end{align*}
To estimate $I_3$, we note that for $x\in D_2$, $y\in D_4$  we have  $|x-y| \geq \frac{1}{2}|y|$, so
\begin{align*}
I_3 &\leq c \int_{D_2} \int_{D_4}  |x-y|^{-d-\alpha} |y|^{-\beta}dy \, dx \leq c_2\int_{D_4}  |y|^{-d-\beta-\alpha} dy \, dx  <\infty.
\end{align*}
\end{proof}
\begin{proof}[Proof of Theorem~\ref{thm:Theorem2}]
By Theorem \ref{thm:Theorem1} and \eqref{e:ip}, $\EEE_p[u] \geq \kappa_{(d-\alpha)/p} \int_{\Rd} |u(x)|^p|x|^{-\alpha} dx$  for $u\in L^p(dx)$. Let $\kappa>\kappa_{(d-\alpha)/p}$. It suffices to verify that
\begin{equation}\label{e.onH}
\EEE_p[u] < \kappa \int_{\Rd} \frac{|u(x)|^p}{|x|^{\alpha}} dx<\infty
\end{equation}
for some $u \in L^p(\Rd, (1 \vee |x|^{-\alpha}) dx)$.

We first consider the case $p>2$.
Let $(d-\alpha)/p<\beta<(p-1)(d-\alpha)/p$.
We note that $\kappa_{\beta} > \kappa_{(d-\alpha)/p}$.
As usual, let $h(x)=h_\beta(x)= |x|^{-\beta}.$
Let $\mu>\beta \vee d/p$, $R>1$ and
\begin{equation}\label{e.mi}
u_R(x) =  \left|x \right|^{-\beta/(p-1)} \wedge \left|x \right|^{-\beta} \wedge R^{\mu -\beta} |x|^{-\mu},\quad x\in \Rd,
\end{equation}
comp. \eqref{e.mi2}. This will be our test function for the Hardy inequality \eqref{e.Hp}.
 We have
$u_R \in L^p(\Rd, (1+ |x|^{-\alpha}) dx)$. Therefore not only the Hardy inequality holds true for $u_R$, but also the representation from Lemma \ref{thm:lem4} is valid:
\begin{align*}
&\mathcal{E}_p[u_R] = \kappa_{\beta} \int_{\Rd} \frac{\left|u_R(x)\right|^p}{\left|x\right|^{\alpha}} dx \\
 &+ \underset{t\to 0}\lim \ \frac{1}{2}\int_{\Rd} \int_{\Rd} \left( \frac{u_R(x)}{h(x)}-\frac{u_R(y)}{h(y)} \right) \left( \frac{u_R(x)^{\langle p - 1 \rangle} }{h(x)}-\frac{u_R(y)^{\langle p - 1 \rangle} }{h(y)} \right) h(x) h(y) \frac{p_t(x,y)}{t} \,dy \,dx,
\end{align*}
and
the first integral on the right-hand side is finite. We next show that the limit  of the above double integral can be made negative by choosing a sufficiently large $R.$
Let
$B_1=B(0,1)$,
$B_2=B_2^R=B(0,R)\setminus B(0,1)$, $B_3 =B_3^R= B(0,R)^c$. By symmetry, we can split the integral as follows:
\begin{align*}
\int_{\Rd}\int_{\Rd} &= \int_{B_1}\int_{B_1} + \int_{B_2}\int_{B_2} + \ 2 \int_{B_1}\int_{B_2} + \int_{B_3}\int_{B_3} + \ 2 \int_{B_1}\int_{B_3} + \ 2\int_{B_2}\int_{B_3} \\
&=: I_1 + I_2 + 2I_3^R + I_4^R + 2I_5^R + 2I_6^R.
\end{align*}
Note that $u_R(x)=|x|^{-\beta/(p-1)}$ on $B_1$, $u_R(x)=|x|^{-\beta}$ on $B_2$ and $u_R(x)=R^{\mu-\beta}|x|^{-\mu}$ on $B_3$.
  We see that $I_1=I_2=0$  and the the integrand in $I_3^R$  is negative.
Moreover, observe that $I_3^R$ decreases when $R$ increases, so
there is a constant $A_0>0$ such that
\begin{align*}
&\limsup_{t\to 0} \int\limits_{B_1}\int\limits_{B_2^R} \left( \frac{u_R(x)}{h(x)}-\frac{u_R(y)}{h(y)} \right) \left( \frac{u_R(x)^{\langle p - 1 \rangle} }{h(x)}-\frac{u_R(y)^{\langle p - 1 \rangle} }{h(y)} \right) h(x) h(y) \frac{p_t(x,y)}{t} \,dy \,dx\\
& <-A_0
\end{align*}
for all  $R>1$. We also have
\begin{align*}
|I_4^R| &\leq  C\int_{B_3} \int_{B_3} R^{p(\mu-\beta)} \left| |x|^{\beta-\mu} - |y|^{\beta-\mu} \right| \left||x|^{\beta-(p-1)\mu}-|y|^{\beta-(p-1)\mu}\right| |x|^{-\beta} |y|^{-\beta} \nu(x,y) \, dy \, dx \\
&=  C R^{d-\alpha-p\beta} \int_{B_1^c}\int_{B_1^c}\left| |x|^{\beta-\mu} - |y|^{\beta-\mu} \right| \left||x|^{\beta-(p-1)\mu}-|y|^{\beta-(p-1)\mu}\right| |x|^{-\beta}|y|^{-\beta} \nu(x,y) \, dy \, dx \\
&\to 0 \text{ as } R\to \infty,
\end{align*}
see Lemma \ref{thm:lem5}.
Furthermore,
\begin{align*}
|I_5^R| &\leq C\int_{B_1}\int_{B_3} \left| |x|^{(p-2)\beta/(p-1)} - R^{\mu-\beta} |y|^{\beta-\mu} \right| \left|1-R^{(p-1)(\mu-\beta)}|y|^{\beta-(p-1)\mu}\right| |x|^{-\beta}|y|^{-\beta}\nu(x,y) \, dy \, dx \\
& \leq C\int_{B_1}\int_{B_3} |x|^{-\beta/(p-1)} |y|^{-\beta} \nu(x,y) \, dy \, dx +C \int_{B_1} \int_{B_3} |x|^{-\beta}|y|^{\beta-p \mu} R^{p(\mu-\beta)} \nu(x,y) \, dy \, dx \\
&\leq c\int_{B_3} |y|^{-\beta-d-\alpha} dy + cR^{p(\mu-\beta)} \int_{B_3} |y|^{\beta-p\mu-d-\alpha} dy \\
& = c_1 \left(R^{-\beta-\alpha} + R^{-(p-1)\beta-\alpha}\right)\to 0 \text{ as } R\to \infty.
\end{align*}
Finally,
\begin{align*}
|I_6^R| &\leq C\int_{B_2}\int_{B_3} \left| 1-R^{\mu-\beta}|y|^{\beta-\mu}\right| \left||x|^{-(p-2)\beta} - R^{(p-1)(\mu-\beta)}|y|^{\beta-(p-1)\mu}\right| |x|^{-\beta} |y|^{-\beta} \nu(x,y) \, dy \, dx  \\
&= CR^{d-\alpha-p\beta}\int_{A_R}\int_{B_1^c} \left| 1-|y|^{\beta-\mu}\right| \left||x|^{-(p-2)\beta}-|y|^{\beta-(p-1)\mu} \right| |x|^{-\beta}|y|^{-\beta} \nu(x,y) \, dy \,dx \\
&\to 0 \text{ as } R\to \infty,
\end{align*}
see Lemma \ref{thm:lem6}.
Hence, for $R$ sufficiently large,
$\EEE_p[u_R] < \kappa_{\beta} \int_{\Rd} |u_R(x)|^p|x|^{-\alpha} dx$.

Since $\beta\mapsto\kappa_{\beta}$ is continuous, symmetric about $(d-\alpha)/2$ and increasing on $\left(0,(d-\alpha)/2\right)$,
there is $\beta \in \left((d-\alpha)/p, (p-1)(d-\alpha)/p \right)$ such that $\kappa_{(d-\alpha)/p} < \kappa_{\beta} < \kappa$. Then $\EEE_p[u_R] < \kappa \int_{\Rd} |u_R(x)|^p |x|^{-\alpha} dx,$ which proves \eqref{e.onH}, and the theorem, for $p>2$.

We now consider $p \in(1,2)$. Let $p'= \frac{p}{p-1}$ be its H\"{o}lder conjugate. Of course, $p' \in (2, \infty)$ and $(p-1)(p'-1)=1$, or $(p-1)p'=p$.
For $u \in L^p(\Rd)$ we let $v=u^{\langle p-1 \rangle}$, i.e., $v^{\langle p'-1\rangle}=u$. We have $|u|^p=|v|^{p'}$,
so $\int_\Rd |v|^{p'} dx=\int_\Rd |u|^p dx$ and $v \in L^{p'}(\Rd)$. Also,
\begin{align}
\mathcal{E}_{p'}[v] &=\frac{1}{2} \int_{\Rd} \int_{\Rd} (v(x)^{\langle p'- 1\rangle} - v(y)^{\langle p' - 1 \rangle})(v(x) - v(y))  \nu(x,y) \, dy \, dx
= \mathcal{E}_p[u]
\end{align}
and
\begin{align}
\int_{\Rd} \frac{|v(x)|^{p'}}{|x|^{\alpha}} dx=\int_{\Rd} \frac{|u(x)|^p}{|x|^{\alpha}} dx.
\end{align}
It follows from the first part of the proof that the Hardy inequality holds for the exponents $p$ and $p'$ with the same constants. In particular, the constant $\kappa_{(d-\alpha)/p}$ is optimal on $L^{p'}(\Rd)$. This conforms with our claim, since
$ \kappa_{(d-\alpha)/p'}=  \kappa_{(d-\alpha)/p}$, as noted in the proof of Lemma~\ref{thm:lem7}.
\end{proof}
The following technical result will be useful in Section~\ref{s.ap}. Its proof may be of independent interest, since it sheds some light on the structure of $\mathcal D(\EEE_p)$.
\begin{lem}\label{l.tm}
If $p>2$ then the inequality \eqref{e.onH} holds for some  function in $C^\infty_c(\Rd)_+$.
\end{lem}
\begin{proof}
Assume that
$\kappa >\kappa_{(d-\alpha)/p}$.
Let $u$ be the function defined by \eqref{e.mi} with suitable $R$
and such that  \eqref{e.onH} holds.
The function is radially decreasing, meaning that $u(x)\le u(y)$ if $|x| \ge |y|$.
For $\epsilon>0$ we let $\phi_\epsilon(a)=(a-\epsilon)\vee 0$, $a\in\mathbb R$.
Of course,
\begin{equation}\label{e.ctr}
|\phi_\epsilon(a)| \leq |a| \qquad \mbox{and}\qquad |\phi_\epsilon(b)-\phi_\epsilon(a)|\leq |b-a|,\quad a,b\in\mathbb R.
\end{equation}
Consequently, by \eqref{eq:2e},
\begin{eqnarray*}
F_p(\phi_\epsilon(u(x)),\phi_\epsilon(u(y))) &\approx& (\phi_\epsilon(u(y))-\phi_\epsilon(u(x)))^2(|\phi_\epsilon(u(y))|+|\phi_\epsilon(u(x))|)^{p-2}\\
&\leq &(u(y)-u(x))^2(|u(y)|+|u(x)|)^{p-2} \\
&\approx & F_p(u(x),u(y)).
\end{eqnarray*}
Since $\phi_\epsilon(u)\to u$ as $\epsilon\to 0$, from the dominated convergence theorem we get
\[\mathcal E_p[\phi_\epsilon(u)]\to \mathcal E_p[u].\]
By Fatou's lemma
\[\liminf_{\epsilon\to 0}\int\frac{|\phi_\epsilon(u)|^p}{|x|^\alpha}\,dx \ge \int\frac{|u|^p}{|x|^\alpha}\,dx,\]
therefore taking a sufficiently small $\epsilon$ we get that \eqref{e.onH} holds with $\phi_\epsilon(u)$.
Slightly abusing the notation, below we write $u$ for the latter function, so $u:=\phi_\epsilon(u)$.
The function is radially decreasing and compactly supported and
we have $u^{\langle p/2\rangle}\in \mathcal D(\mathcal E),$ which follows from \eqref{e.cHk}.
Let $\varphi:\mathbb R^d\to\mathbb R_+$ be smooth, compactly supported, radially decreasing and such that $\int \varphi=1$.
For $\eta>0$ denote $\varphi_\eta(x)=\eta^{-d}\varphi(x/\eta)$
and let
\[
v_\eta=u^{\langle p/2\rangle}\ast \varphi_\eta.
\]
Each $v_\eta$ is smooth, nonnegative and of compact support. It is also radially decreasing, as a convolution of two such functions, see Beckner \cite[page 171]{MR385456}.
It is evident that $v_\eta\to u^{\langle p/2\rangle}$ (pointwise) as $\eta\to 0$.
By the arguments from \cite[the proof Lemma A.5]{MR4088505} we see that $\mathcal E[v_\eta-u^{\langle p/2\rangle}]\to 0$, hence by
Vitali's theorem, the functions
\[\mathbb R^d\times\mathbb R^d\ni(x,y)\mapsto (v_\eta(x)-v_\eta(y))^2\]
are uniformly integrable with respect to $\nu(x,y)dxdy$,
see Schilling \cite[Theorem 22.7]{MR3644418}.
Let
\begin{equation}\label{eq:approxi}
u_{\eta} = \left(v_\eta\right)^{\langle 2/p\rangle}=\left(u^{\langle p/2\rangle}\ast \varphi_\eta\right)^{\langle 2/p\rangle}.
\end{equation}
Since $v_\eta\ge 0$ is smooth and radially decreasing, $u_{\eta}$ is smooth, despite the fractional exponent in its definition.
It is clear that $u_{\eta}\to u$ pointwise. Since
\[F_p(u_{\eta}(x),u_{\eta}(y))\approx (v_\eta(x)-v_\eta(y))^2,\]
the left-hand side is uniformly integrable with respect to $\nu(x,y)dxdy$.
By Vitali's theorem again, $\mathcal E_p[u_{\eta}]\to\mathcal E_p[u]$ as $\eta\to 0$.
By Fatou's lemma, taking sufficiently small $\eta$ we see that \eqref{e.onH} holds with $u:=u_\eta.$
The proof is complete.
\end{proof}

\section{Application to parabolic equation}\label{s.ap}
In this section we prove for $p\in (1,\infty)$ that the Feynman-Kac semigroup generated by the Schr\"{o}dinger operator $\Delta^{\alpha/2}+\kappa_{(d-\alpha)/p}|x|^{-\alpha}$ is a contraction on $L^p(\Rd)$, and the threshold $\kappa_{(d-\alpha)/p}$ is sharp.
We shall also see that the semigroup generated by
$\Delta^{\alpha/2}+
\kappa |x|^{-\alpha}$ is bounded on $L^p(\Rd)$
if and only if either $d/p^*\ge (d-\alpha)/2$ and $\kappa \le \kappa_{(d-\alpha)/2}$ or
$d/p^*< (d-\alpha)/2$ and $\kappa \le \kappa_{d/p^*}$
Here $p^*=\max\{p,p/(p-1)\}$.

As usual, let $\alpha\in (0,2)$ and $\alpha<d\in \N$. For $\delta\in [0, (d-\alpha)/2]$, $\kappa=\kappa_\delta$ and  $q=q_\delta$ as in \eqref{e:qbeta},  we define (cf. \cite{MR3933622}) the Schr\"odinger perturbation of $p_t$ from \eqref{e.dp} by $q=q_\delta:$
\begin{equation}\label{def_p_tilde}
   \tilde p_t = \sum_{n=0}^{\infty }p_{t}^{(n)}.
\end{equation}
Here for $t>0$ and $x,y\in \Rd$ we let $p_{t}^{(0)}(x, y) = p_t(x, y)$ and
\begin{align}\label{pn}
     p_{t}^{(n)}(x, y) &= \int_0^t \int_{\Rd} p_s(x, z) q(z)p_{t-s}^{(n-1)}(z, y) \, dz ds\\
     &=\int_0^t \int_{\Rd} p_{s}^{(n-1)}(x, z) q(z)p_{t-s}(z, y) \, dz ds
     , \quad n \geq 1.\notag
\end{align}
From the general theory \cite{MR3460023}, $\tilde p_t$ is a symmetric transition density, i.e., the following Chapman-Kolmogorov equation holds:
\begin{align}\label{eq:ChK}
\int_\Rd \tp_s(x,z) \tp_t(z,y)\, dz = \tp_{t+s}(x,y)\,.
\end{align}
The scaling of $\tp_t(x,y)$ is the same as that of $p_t(x,y)$ \cite[Lemma 2.2]{MR3933622}:
\begin{equation}\label{ss form}
  \tilde p_t(x, y)=t^{-\frac{d}{\alpha }}\tilde p_1\big(xt^{-\frac{1}{\alpha }}, yt^{-\frac{1}{\alpha }}\big), \qquad  t>0,  \ x,y \in \Rd.
\end{equation}
For $t>0$ and $x\neq 0$ we define the Feynman-Kac semigroup
$$\tilde P_t f(x) = \int_\Rd \tilde p_t(x, y) f(y) dy.$$
The integral certainly makes sense if $f$ is nonnegative, but the following result in fact shows that $\tilde P_t$ may be contractive on $L^p(\Rd)$.
Before we proceed, we note that for nonnegative functions $f,g$ on $\Rd$ we have
\begin{equation}\label{e.duality}
\int_\Rd\int_\Rd \tilde p_t(x,y)f(y)dy\ g(x)dx=
\int_\Rd\int_\Rd \tilde p_t(x,y)g(y)dy\ f(x)dx.
\end{equation}
By H\"{o}lder inequality, for each $p\in (1,\infty)$ the operator norm of $\tilde P_t$ on $L^p$ is the same as on $L^{p/(p-1)}$ -- below this fact will be referred to as the duality argument.
Furthermore, by \eqref{ss form} we have
\begin{equation}\label{e.es}
\|\tilde P_t\|_{L^p(\Rd)\to L^p(\Rd)}=\|\tilde P_1\|_{L^p(\Rd)\to L^p(\Rd)}, \quad t>0.
\end{equation}
\begin{rem}\label{r.n}
For clarity we note that the heat kernel $\tilde p_t$ of $\Delta^{\alpha/2}+\kappa |x|^{-\alpha}$  can be defined also for  
$\kappa\in (-\infty, 0)$, see Jakubowski and Wang \cite{MR4140086}, see also Cho, Kim, Song and Vondra\v{c}ek \cite{MR4163128}. Then
$0\le \tilde p_t\le p_t$, so $\|\tilde P_t\|_{L^p(\Rd)\to L^p(\Rd)}\le \|P_t\|_{L^p(\Rd)\to L^p(\Rd)}=1$ for every $p\ge 1$. 
Further, for $\kappa>\kappa_{(d-\alpha)/2}$ we have $\tilde p_t(x,y)\equiv \infty$ by \cite[Corollary~4.11]{MR3933622}, so $\|\tilde P_t\|_{L^p(\Rd)\to L^p(\Rd)}=\infty$ in this case. The remaining cases are resolved in 
Theorem ~\ref{t.znHp} and \ref{t.bd}, by considering $\kappa=\kappa_\delta$  with $\delta\in [0, (d-\alpha)/2]$.
\end{rem}
We recall that analogous perturbations of local elliptic operators have been widely investigated.
We refer, e.g., to Kovalenko, Perelmuter and Semenov \cite{MR622855}. In particular, the approach covers Hardy-type perturbations of the classical Laplacian. See also
\cite{MR4043021}.
 The relation of the Hardy inequality to the contractivity of the corresponding Feynman-Kac semigroups is also the subject of \cite{MR2259099}. In
\cite[Corollary 1.2]{MR2259099} the authors use the Hardy inequality to prove that
 the semigroup corresponding to the operator $-\Delta+\kappa|x|^{-2}$
is contractive on $L^p(\Rd)$ for $d\ge 2$ if and only if $\kappa\le (d-2)^2(p-1)/p^2$. This accords well with Theorem~\ref{t.znHp} because
$(d-2)^2(p-1)/p^2=\kappa_{(d-\alpha)/p}$, if we let $\alpha=2$.

Let us first present an informal idea of the proof of Theorem~\ref{t.znHp}.
Consider
\begin{align}\label{eq:CP}
\begin{cases}
u_t = \Delta^{\alpha/2} u + \kappa |x|^{-\alpha} u, \qquad t>0, x \in \R^d,\\
u(0,x) = f(x), \qquad x \in \R^d.
\end{cases}
\end{align}
The semigroup solution of this Cauchy problem is $u(t,x)=\tilde P_t f(x)$, at least for $f$ in the domain of the generator.
We multiply both sides of \eqref{eq:CP} by $u^{\langle p-1\rangle}$ and integrate
\begin{align*}
\int_{\R^d} u_t u^{\langle p-1\rangle} dx = \int_{\R^d} u^{\langle p-1\rangle} \Delta^{\alpha/2}u dx + \kappa \int_{\R^d} |u|^p |x|^{-\alpha} dx.
\end{align*}
By calculus and \eqref{e.Hp},
\begin{align}\label{e.znH}
\frac{\partial}{\partial t}\int_{\R^d} \frac{|u|^p}{p}  dx \le  \left(- \kappa_{(d-\alpha)/p}+\kappa\right) \int_{\R^d} |u|^p |x|^{-\alpha} dx \le 0,
\end{align}
hence $\|u(t,\cdot)\|_p^p$ is decreasing, and so $\|\tilde P_t f\|_p \le \|f\|_p$. On the other hand, for $\kappa>\kappa_{(d-\alpha)/p}$ we have the opposite inequality for the function from Lemma~\ref{l.tm}.
We now present rigorous arguments (some technical results used in the proof are still deferred to Section~\ref{s.A}).
\begin{proof}[Proof of Theorem~\ref{t.znHp}]
In the case of $p=2$, the condition $\delta\le (d-\alpha)/2$ is clearly met, while the contractivity of $\tilde P_t$ is proved in \cite[Proposition 2.4]{MR3933622}, using Schur's test.
Let $p\in (2,\infty)$.
 For $M\in (0,\infty)$ let $q^{(M)}(x)=q(x)\wedge M$, where, recall, $q=q_\delta$.
 In a similar manner as above, we define the Schr\"odinger perturbation of $p_t$ by $q^{(M)}$, which we denote $\tilde p^{(M)}_t$. Since $q^{(M)}$ is bounded, by Phillips' perturbation theorem the domain of the generator $\Delta^{\alpha/2}$ of the strongly continuous semigroup $P_t$ on $L^p(\Rd)$ is the same as the domain of the generator, $\Delta^{\alpha/2}+q^{(M)}$, of the strongly continuous semigroup
$$\tilde P^{(M)}_t f(x) := \int_\Rd \tilde p^{(M)}_t(x, y) f(y) dy.$$
We refer to Phillips \cite[Theorem~3.2]{MR54167} and \cite[Lemma 4.2]{MR3613319} for details.
Let $f$ be in the domain of $\Delta^{\alpha/2}$ on $L^p(\Rd)$.
Let $u^{(M)}(t,x)=\tilde P^{(M)}_t f(x)$.
Denote $u^{(M)}(t)=u^{(M)}(t,\cdot)$.
From the general semigroup theory, the mapping $[0,\infty)\ni t\mapsto u^{(M)}(t)\in L^p(\Rd)$ is continuously differentiable.
We then verify that $t \mapsto |u^{(M)}(t)|^p$ and $t \mapsto u^{(M)}(t)^{\langle p - 1 \rangle}$ are continuous in $L^1(\Rd)$ and $L^{\frac{p}{p-1}}(\Rd)$, respectively, see Lemma \ref{thm:difpow0}.
Then it follows from Lemma \ref{lem:pdif} that $|u^{(M)}(t)|^p$ is continuously differentiable in $L^1(\Rd)$ and
\begin{align}\label{eq:diff}
\frac{d}{dt}\int_\Rd |u^{(M)}(t)|^p dx&=
\int_\Rd \frac{d}{dt}|u^{(M)}(t)|^p dx=\int_\Rd p u^{(M)}(t)^{\langle p-1 \rangle}\frac{d}{dt}u^{(M)}(t) dx \nonumber \\
&=
p\int_\Rd u^{(M)}(t)^{\langle p-1 \rangle}(\Delta^{\alpha/2} u^{(M)}(t)+q^{(M)}u^{(M)}(t)) dx\nonumber \\
&= p\left(-\mathcal E_p[u^{(M)}(t)] +\int_{\mathbb R^d}q^{(M)}|u^{(M)}(t)|^p dx\right)
\le 0,
\end{align}
provided $\kappa\in [0,\kappa_{(d-\alpha)/p}]$, see \eqref{e.znH}. We then get $\int_\Rd |\tilde P_t^{(M)} f|^pdx\le \int_\Rd |f|^p dx$. This extends to all $f\in L^p(\Rd)$ by the density of the domain of the generator.
We then let $M\to \infty$. By monotone convergence,  $\tilde p_t^{(M)}\uparrow \tilde p_t$ and for every nonnegative $f\in L^p(\Rd)$,
 \[\int_\Rd (\tilde P_t f)^p dx\leftarrow \int_\Rd (\tilde P^{(M)}_t f)^p dx\le \int_\Rd f^p dx.\]
The estimate for general $f\in L^p(\Rd)$ follows by considering $f=f_+-f_-$, see \eqref{eq:fgen}.

The case of $p\in (1,2)$ results from duality with the same range of $\kappa$ as for the H\"{o}lder conjugate of $p$, that is
the contractivity of $\tilde P_t$ holds for $\kappa\in[0,\kappa_{(d-\alpha)/p'}]$ with $p'=p/(p-1)$. But $\kappa_{(d-\alpha)/p}=\kappa_{(d-\alpha)-(d-\alpha)/p} = \kappa_{(d-\alpha)/p'}$. Thus, the contractivity of $\tilde P_t$ on $L^p(\Rd)$ is proved if $p\in (1,\infty)$ and $\kappa\le \kappa_{(d-\alpha)/p}$.

We next assume $\kappa >\kappa_{(d-\alpha)/p}$ and prove that the contractivity fails. To this end
we first consider $p\in (2,\infty)$.
By Lemma~\ref{l.tm}
there is a nonnegative $f\in C_c^\infty(\mathbb R^d)$ such that
\begin{equation}\label{e.onHf}
\EEE_p[f] < \kappa \int_{\Rd} \frac{|f(x)|^p}{|x|^{\alpha}} dx<\infty.
\end{equation}
It is well known that $C_c^\infty(\mathbb R^d)$ is a subset of the domain of the generator of the semigroup $\{P_t,t\ge 0\}$ acting on $L^p$, see, e.g., Jacob \cite[Theorem 3.3.11]{MR1917230} or Farkas, Jacob and Schilling \cite[Theorem 1.4.2 or Proposition 2.1.1]{MR1840499}; or use \cite[Theorem~31.5]{MR1739520}, \eqref{e.mp} and \eqref{e.zp}. Therefore $f$ is in the domain of $\Delta^{\alpha/2}$ on $L^p(\Rd)$.
For $M\in (0,\infty)$ we write $u^{(M)}(t,x)=\tilde P_t^{(M)} f(x)$ or $u^{(M)}(t)=\tilde P_t^{(M)} f$.
By \eqref{eq:diff},
\begin{eqnarray*}
\frac{d}{dt}\int_{\mathbb R^d} |u^{(M)}(t)|^p dx &=&p\int_{\mathbb R^d} u^{(M)}(t)^{\langle p-1\rangle}\left(\Delta^{\alpha/2}u^{(M)}(t) +q^{(M)} u^{(M)}(t)\right)\,dx\\
&=& p\left(-\mathcal E_p[u^{(M)}(t)] +\int_{\mathbb R^d}q^{(M)}|u^{(M)}(t)|^p dx\right).
\end{eqnarray*}
In particular,
\begin{eqnarray}\label{eq:contr1}
&&\left.\frac{d}{dt}\int_{\mathbb R^d} |u^{(M)}(t)|^p dx \right|_{t=0}
=p\left(-\mathcal E_p[f] +\int_{\mathbb R^d} q^{(M)}|f|^pdx\right)\nonumber
\\
&=& p\left(\int_{\mathbb R^d}\left(q^{(M)}(x)-\frac{\kappa}{|x|^\alpha}\right)
|f|^pdx\right)
+p\left(\kappa\int_{\mathbb R^d} \frac{|f|^p}{|x|^\alpha}dx - \mathcal E_p[f]\right).
\end{eqnarray}
The last term is strictly positive by \eqref{e.onHf}. Since $q^{(M)}(x) \to\kappa|x|^{-\alpha}$, for sufficiently large $M$ the expression in \eqref{eq:contr1} is strictly positive, and so for such $M$ we get,
\[\frac{d}{dt}\|u_\kappa^{(M_0)}(t)\|_p^p\bigg|_{t=0} >0.\]
Therefore for small $t>0$ we have
\[ \|f\|_p^p= \|u_\kappa^{(M)}(0)\|_p^p < \|u_\kappa^{(M)}(t)\|_p^p\le \|\tilde P_t f\|_p^p.\]
The case $p\in (1,2)$ follows by the duality argument.
\end{proof}

To complement Theorem~\ref{t.znHp} we note that for $x,y\in \Rd$ and $t>0$,
\begin{equation}
\label{eq:mainThmEst}
\tilde{p}_t(x,y)\approx \left(1+t^{\delta/\alpha}|x|^{-\delta}  \right)\left(1+t^{\delta/\alpha}|y|^{-\delta}  \right) \left( t^{-d/\alpha}\wedge \frac{t}{|x-y|^{d+\alpha}} \right).
\end{equation}
The result is given in \cite[Theorem~1]{MR3933622} (for $x=0$ or $y=0$ we have $p^{(1)}_t(x,y) = \infty$; see \cite[Lemma 3.7 and the comment before Corollary 3.9]{MR4140086}, so $\tilde{p}_t(x,y) = \infty$).
Denote $H(x)=|x|^{-\delta}\vee 1$, $x\in \Rd$.
Clearly,
\begin{equation}\label{e.oHH}
\tilde p_1(x,y)\approx H(x)H(y)p_1(x,y), \quad x,y\in \Rd.
\end{equation}
\begin{proof}[Proof of Theorem~\ref{t.bd}]
By \eqref{e.es} it is enough to consider $t=1$.
Let $B=B(0,1)\subset \Rd$. Denote $q=p/(p-1)$. If $\delta\ge d/q$, then $\int_B H^q=\infty$, so there is $f\in L^p_+$ such that $\int_B H(y)f(y)=\infty$, see, e.g., \cite[Corollary~4.4.5]{MR2267655}. By \eqref{e.oHH}, $\tilde P_1 f=\infty$ on $\Rd$.
Also, if $\delta\ge d/p$, then $\tilde P_t \indyk_B\ge c H\indyk_B$, so $f:=\indyk_B\in L^p$ but $\int |\tilde P_t f|^p\ge c \int_B H^p=\infty$.
If $0\le \delta<d/p^*$, then we let $f\in L^p_+$, $g\in L^q_+$ and consider
\begin{align*}
I(f,g)&:=\int_\Rd\int_\Rd g(x)H(x)H(y)p_1(x,y)f(y)dydx\\
&=I(f\indyk_B,g\indyk_{B})+I(f\indyk_B,g\indyk_{B^c})+I(f\indyk_{B^c},g\indyk_{B})+I(f\indyk_{B^c},g\indyk_{B^c}).
\end{align*}
By \eqref{e.cLp} and H\"{o}lder inequality,
$$
I(f\indyk_{B^c},g\indyk_{B^c})\le \|f\|_p\|g\|_q.
$$
By H\"{o}lder inequality, $\int_B gH\le c\|g\|_q$, where $c=\big(\int_B |x|^{-\delta p}\big)^{1/p}<\infty$. Similarly,
$\int_B fH\le c\|f\|_p$, where $c=\big(\int_B |x|^{-\delta q}\big)^{1/q}<\infty$.
Furthermore, since $p_1$ is bounded, $I(f\indyk_B,g\indyk_{B})\le c \int_B gH \int_B f H\le c\|g\|_q\|f\|_p$.
By \eqref{e.cLi},
$I(f\indyk_{B^c},g\indyk_{B})\le \int_B gH \|p_1f\|_\infty\le c\|g\|_q\|f\|_p$. Similarly, $I(f\indyk_B,g\indyk_{B^c})\le c\|g\|_q\|f\|_p$.
By \eqref{e.oHH}, $\tilde P_1$ is bounded on $L^p$.
\end{proof}

\section{Appendix}\label{s.A}
For the convenience of the reader, we add some details to the arguments used in the proof of Theorem \ref{t.znHp}.
In what follows, we consider a measure space $(X,\mu)$ and the spaces $L^p(X,d\mu)$ for $p\geq 1$.
For simplicity, we write $L^p$ for $L^p(X,d\mu).$
Recall that $p/(p-1)$ is the H\"{o}lder conjugate exponent of $p\in (1,\infty)$ (i.e. we have $p^{-1} + (p/(p-1))^{-1} = 1$).
\begin{lem}[Continuity]\label{thm:continuity}
Let $p \in (1, \infty)$ and $r \in \big[\frac{p}{p-1}, \infty \big)$.
If $f\in L^p$, $g\in L^r$, then $\| fg\|_{\frac{pr}{p+r}} \leq \|f\|_p \|g\|_r$. If $f_n \to f$ in $L^p$ and $g_n \to g$ in $L^r$, then $f_ng_n \to fg$ in $L^{\frac{pr}{p+r}}$.
\end{lem}
\begin{proof}
Of course, $r>1$. Also, $\frac{p+r}{pr} = \frac{1}{r} + \frac{1}{p} \leq \frac{p-1}{p} + \frac{1}{p} = 1$, thus $\frac{pr}{p+r} \in [1, \infty)$. For $f \in L^p$, $g \in L^r$, by H\"older inequality with exponents $\frac{p+r}{r}$ and $ \frac{p+r}{p}$,
$$
\int |fg|^{\frac{pr}{p+r}} \leq \left(\int |f|^p \right)^{\frac{r}{p+r}} \left(\int |g|^r \right)^{\frac{p}{p+r}} < \infty,
$$
and we get the first statement.
The second statement is verified as follows,
\begin{align*}
\|f_ng_n - fg\|_{\frac{pr}{p+r}} &= \|f_ng_n - f_ng +f_ng -fg\|_{\frac{pr}{p+r}} \\
&\leq \|f_n\|_p \|g_n-g\|_r + \|f_n - f \|_p \|g\|_r \to 0,
\end{align*}
as the sequence $f_n$ is bounded in $L^p.$
\end{proof}
Let $p \in (1,\infty)$. Assume that
$[0,\infty) \ni t \mapsto u(t) \in L^p$.
We will relate the continuity and differentiability properties of $u$ in $L^p$ to those of $|u|^p$ in $L^1$.
We denote
$$\Delta_hu(t) = u(t+h)-u(t), \quad \textnormal{if} \,\, t, t+h\ge 0.$$

We say that $u$ is continuous in $L^p$ at $t\ge 0,$ if $\Delta_hu(t) \to 0$  in $L^p$ as $h\to 0$,
and we say $u$ is continuously differentiable at $t\ge 0$ if  $u'(t) := \lim_{h \to 0} \frac{1}{h} \Delta_hu(t)$ in $L^p$ with
continuous $u'$. Of course, $u'(0)$ is the right-hand side derivative in the above setting.

\begin{lem}\label{thm:difpow0}
If $u$ is continuous in $L^p$, then $|u|^p$  and $u^{\langle p - 1 \rangle}$ are continuous in $L^1$ and $L^{\frac{p}{p-1}}$, respectively.
\end{lem}
\begin{proof}
By \eqref{eq:2c},
$|\Delta_h |u(t)|^p|\le (p+C)|\Delta_h u(t)|(|u(t+h)|+|u(t)|)^{p-1}$.
Similarly, by \eqref{eq:2dd}, $|\Delta_h u(t)^{\langle p - 1 \rangle}|\le c'|\Delta_h u(t)|^\lambda (|u(t+h)|+|u(t)|)^{p-1-\lambda}$, and we can pick $\lambda>0$ such that $p-1-\lambda>0$.
The results follow from Lemma \ref{thm:continuity}.
\end{proof}

\begin{lem}[Differentiability]\label{thm:difpow}
If $[0,\infty)\ni t \mapsto u(t)$ is continuously differentiable in $L^p$,
then $|u|^p$ is continuously differentiable in $L^1$ and $(|u|^p)' = pu^{\langle p - 1 \rangle} u'$.
\end{lem}
For comparison, Marinelli, R\"ockner \cite[p. 4]{MR3463679} assert that for $p \geq 2$ the function
$\phi: H \ni x \mapsto \|x\|^p$ is weekly differentiable for every Hilbert space $H$, with the Fr\'echet derivative
$D\phi(x): y \mapsto p \|x\|^{p-2} (x,y)$. 
\begin{proof}[Proof of Lemma \ref{thm:difpow}]
By Lemma~\ref{thm:continuity},
$u^{\langle p - 1 \rangle} u'$ is continuous in $L^1$.
By \eqref{eq:2d}, for
$\lambda\in [0,2]$,
\begin{align*}
\left|\frac{1}{h} \left(\Delta_h |u|^p (t) - p u(t)^{\langle p - 1 \rangle} \Delta_h u(t)\right)\right|
&\leq Ch^{\lambda-1} |\frac{1}{h} \Delta_hu(t)|^{\lambda} (|u(t+h)| +|u(t)|)^{p-\lambda}.
\end{align*}
We pick
$\lambda>1$ such that $p-\lambda>0$ and the result follows.
\end{proof}

Recall that $(P_t, t \geq 0)$ is a strongly continuous operator semigroup on $L^p$. Let $f$ be in the domain of its generator $A$. Let $u(t) = P_tf$. Then $u'(t) = P_tAf = AP_tf = Au(t)$.
By Lemma \ref{thm:difpow} we obtain the following result.
\begin{lem}\label{lem:pdif}
$|u(t)|^p$ is differentiable in $L^1$ with the derivative
\begin{equation}\label{e.ppp}
(|u(t)|^p)' =pu(t)^{\langle p - 1 \rangle} u'(t)= pu(t)^{\langle p - 1 \rangle} P_t Af, \quad t \geq 0.
\end{equation}
\end{lem}
Finally, recall that in the proof of Theorem \ref{t.znHp} we only proved the contractivity of the semigroup for $\kappa\le \kappa_{(d-\alpha)/p}$ and \textit{nonnegative} $f \in L^p$. This suffices because for general $f \in L^p$ we may write $f=f_{+} - f_{-}$ and
we have
$
(\tilde{P}_t f)_+ \leq \tilde{P}_t f_+,
$
and
$
(\tilde{P}_t f)_- \leq \tilde{P}_t f_-,
$
hence
$
|\tilde{P}_t f|^p \leq (\tilde{P}_tf_{+})^{p} + (\tilde{P}_tf_{-})^{p}.
$
Therefore,
\begin{align}
\int |\tilde{P}_tf(x)|^p dx &\leq \int \tilde{P}_tf_+(x)^p dx + \int \tilde{P}_tf_-(x)^p dx \label{eq:fgen}\\
&\leq \int f_+(x)^p dx + \int f_-(x)^p dx = \int |f(x)|^p dx. \nonumber
\end{align}


\end{document}